\documentclass{amsart}
\usepackage{graphicx}
\usepackage{latexsym}
\usepackage{hyperref}
\usepackage{amsfonts}
\usepackage[all]{xy}
\usepackage{amssymb, mathrsfs, amsfonts, amsmath}
\usepackage{amsbsy}
\usepackage{amsfonts}
\usepackage{setspace}
\usepackage{cancel}
\usepackage{xcolor}

\usepackage{mathrsfs}
\usepackage{amsbsy}

\newtheorem{corollary}{Corollary}
\newtheorem{definition}{Definition}
\newtheorem{lemma}{Lemma}
\newtheorem{proposition}{Proposition}
\newtheorem{remark}{Remark}
\newtheorem{theorem}{Theorem}

\numberwithin{equation}{section}

\newcommand{\Rc}{\mathring{Ric}}

\makeatletter
\@namedef{subjclassname@2020}{%
  \textup{2020} Mathematics Subject Classification}
\makeatother

\begin{document}
	
	\title[quasi-Einstein manifolds with boundary]{Rigidity of compact quasi-Einstein\\ manifolds with boundary}

	\author{Johnatan Costa,\,\,Ernani Ribeiro Jr\,\,and\,\,Detang Zhou}
	
\address[J. Costa]{Universidade Federal do Cear\'a - UFC, Departamento  de Matem\'atica, Campus do Pici, Av. Humberto Monte, Bloco 914, 60455-760, Fortaleza - CE, Brazil.}\email{johnatansc@alu.ufc.br}
	
	\address[E. Ribeiro Jr]{Universidade Federal do Cear\'a - UFC, Departamento  de Matem\'atica, Campus do Pici, Av. Humberto Monte, Bloco 914, 60455-760, Fortaleza - CE, Brazil}\email{ernani@mat.ufc.br}

	\address[D. Zhou]{Universidade Federal Fluminense - UFF, Instituto de Matem\'atica e Estat\'istica, 24020-140, Niter\'oi - RJ, Brazil}
\email{zhoud@id.uff.br}

\thanks{J. Costa was partially supported by CAPES/Brazil - Finance Code 001.}	
	\thanks{E. Ribeiro was partially supported by CNPq/Brazil [305128/2025-6 and 351492/2025-9] and FUNCAP/Brazil [ITR-0214-00116.01.00/23].}
	
	\thanks{D. Zhou was partially supported by FAPERJ/Brazil [E-26/200.386/2023] and CNPq/Brazil [308067/2023-1].}
	
	\thanks{Corresponding Author: E. Ribeiro Jr (ernani@mat.ufc.br)}

	\begin{abstract}
		In this article, we investigate the geometry of compact quasi-Einstein ma\-ni\-folds with boundary. We show that a $3$-dimensional simply connected compact quasi-Einstein manifold with boundary and cons\-tant scalar curvature is isometric, up to scaling, to either the standard hemisphere $\mathbb{S}^{3}_{+}$, or the cylinder $I\times\mathbb{S}^2$ with the product metric. For dimension $n=4,$ we prove that a $4$-di\-men\-sio\-nal simply connected compact quasi-Einstein manifold with boundary and constant scalar curvature is isometric, up to scaling, to either the standard hemisphere $\mathbb{S}^{4}_{+},$ or the cylinder $I\times\mathbb{S}^3$ with the product metric, or the product space $\mathbb{S}^{2}_{+}\times\mathbb{S}^2$ with the product metric. Other related results for arbitrary dimensions are also discussed.
		\end{abstract}
	
	\date{November 14, 2025}
	
	\keywords{quasi-Einstein manifolds; constant scalar curvature; compact
		manifolds with boundary; rigidity results}
	\subjclass[2020]{Primary 53C23, 53C24, 53C25; Secondary 58J90.}
	
	\maketitle
	
\section{Introduction}\label{SecInt}

	 A compact $n$-dimensional Riemannian ma\-ni\-fold $(M^n,\,g),$ $n\geq 2,$ possibly with boundary $\partial M,$ is called an $m$-{\it quasi-Einstein manifold}, or simply {\it quasi-Einstein manifold}, if there exists a smooth potential function $u$ on $M^n$ satisfying the system
\begin{equation}
\label{eq-qE}
\left\{%
\begin{array}{lll}
    \displaystyle \nabla^{2}u = \dfrac{u}{m}(Ric-\lambda g) & \hbox{in $M,$} \\
    \displaystyle u>0 & \hbox{on $int(M),$} \\
        \displaystyle u=0 & \hbox{on $\partial M,$} \\
    \end{array}%
\right.
\end{equation} for some constants $\lambda$ and $0<m<\infty$ (cf. \cite{CaseShuWei, He-Petersen-Wylie2012,Petersen-Chenxu}).  Here, $\nabla^{2} u$ stands for the Hessian of $u$ and $Ric$ is the Ricci tensor of $g.$ When $m=1$, we assume in addition that $\Delta u=-\lambda u$ in order to recover the {\it static equation}: $-(\Delta u)g+\nabla^2 u -uRic=0.$ Moreover, an $m$-quasi-Einstein manifold will be called {\it trivial} if $u$ is constant, otherwise it will be {\it nontrivial}. We notice that the triviality implies that $M^n$ is an Einstein manifold.

The study of quasi-Einstein manifolds is directly related to the existence of warped pro\-duct Einstein metrics on a given manifold. To be precise, as discussed by Besse \cite[p. 267]{Besse}, an $m$-quasi-Einstein manifold corresponds to a base of a warped product Einstein metric; for more details, see, e.g., \cite[Proposition 1.1]{He-Petersen-Wylie2012}, \cite[Corollary 9.107]{Besse} and \cite{Ernani2,Besse,CaseShuWei,catino,CMMM,MR,Rimoldi}. If $\partial M=\emptyset,$ we can make sense of $\infty$-quasi-Einstein manifolds by setting $u=e^{-\frac{f}{m}}$ in (\ref{eq-qE}) and taking the limit $m\to \infty.$ These are precisely gradient Ricci solitons; see \cite{Cao1,CaseShuWei,CRZJGEA,Hamilton}. Although quasi-Einstein manifolds and gradient Ricci solitons share structural similarities, there exist examples that exhibit fundamental differences, as discussed in, e.g., \cite[Remark 1.4]{He-Petersen-Wylie2012} and \cite{BRS14,Bohm,CaseShuWei}. Another interesting motivation to investigate quasi-Einstein manifolds comes from the study of diffusion operators by Bakry and \'Emery \cite{bakry}, which is linked to the theory of smooth metric measure spaces; see, e.g., \cite{BRR,CaseT,CaseP,MR,Rimoldi,LFWang,LFWang2,WW,WW2} and the references therein. $1$-quasi-Einstein manifolds are commonly known as static spaces. More precisely, static spaces can be viewed as the relativistic interpretation of $1$-quasi-Einstein manifolds that serve as bases of Einstein manifolds; see \cite[Remark 2.3]{CaseShuWei} and \cite{Ambrozio,BM,BMC,Kobayashi,lafontaine,QY1,QY}. Additionally, quasi-Einstein metrics have attracted interest in physics due to their relation with the geometry of a degenerate Killing horizon and horizon limit; see, e.g., \cite{BGKW,BGKW2,Wylie}. Explicit examples of nontrivial compact and noncompact $m$-quasi-Einstein manifolds can be found in, e.g., \cite{Bergery,Besse,Bohm,Bohm2,CaseShuWei,CaseT,CaseP,He-Petersen-Wylie2012,Lu,Ernani_Keti,Rimoldi,LFWang}. Moreover, the classification of $1$ and $2$-dimensional $m$-quasi-Einstein manifolds is presented in \cite[p. 267-272]{Besse} and \cite{He-Petersen-Wylie2012}.

In this article, we focus on nontrivial compact $m$-quasi-Einstein manifolds with non-empty boundary $\partial M.$ According to \cite[Theorem 4.1]{He-Petersen-Wylie2012}, such manifolds necessarily satisfy $\lambda>0.$ In order to set the stage for our main results, it is important to highlight some examples of compact $m$-quasi-Einstein manifolds with boundary and constant scalar curvature (cf. \cite{He-Petersen-Wylie2012,remarks}):
\begin{itemize}
\item[(i)]
The hemisphere $\Bbb{S}^n_+$ with the standard metric $g=dr^2+\sin^2r g_{\Bbb{S}^{n-1}}$ and  potential function $u(r)=\cos r,$ where $r$ is a height function with $r\leq\frac{\pi}{2};$ 
\item[(ii)] $\Big[0,\sqrt{m/\lambda}\,\pi\Big]\times\Bbb{S}^{n-1},$ for $\lambda>0,$ endowed with the metric $g=dt^2+\frac{n-2}{\lambda}g_{\Bbb{S}^{n-1}}$ and potential function $u(t,x)=\sin\left(\sqrt{\lambda/m}\,t\right);$ 

\item[(iii)]  $\Bbb{S}^{p+1}_+\times\Bbb{S}^q$, $q>1$, with the product metric $$g=dr^2+\sin^2r g_{\Bbb{S}^p}+\frac{q-1}{p+m}g_{\Bbb{S}^q},$$ where $r(x,y)=h(x)$ and $h$ is a height function on $\Bbb{S}^{p+1}_+,$ potential function  $u=\cos r$ with $r\leq\frac{\pi}{2}$ and $\lambda=p+m.$ 
\end{itemize}

In 2014, He, Petersen and Wylie \cite[Proposition 2.4]{Petersen-Chenxu} showed that a nontrivial compact quasi-Einstein manifold with boundary and constant Ricci curvature is isometric to Example ${\rm (i)}.$ It turns out that these three quoted examples have constant scalar curvature. Therefore, one question that naturally arises is to know {\it whether a nontrivial compact (simply connected) $m$-quasi-Einstein manifold with boundary and constant scalar curvature must be necessarily one of them}\footnote{For dimensions $n\geq 5,$ additional examples can be constructed by applying the product property; see, e.g., \cite[Lemma 2.2]{Petersen-Chenxu}.}. As we shall see later, in this article, we will solve this question for dimension $3$ and $4.$

It is known from \cite{Petersen-Chenxu} and \cite[p. 271]{Besse} that the hemisphere $\Bbb{S}^2_+$ is the only nontrivial $2$-dimensional simply connected compact $m$-quasi-Einstein manifold with boundary and constant scalar curvature. In \cite{Petersen-Chenxu}, He, Petersen and Wylie investigated $m$-quasi-Einstein manifolds with constant scalar curvature. In particular, for the specific dimension $n=3,$ they proved that an $m$-quasi-Einstein manifold with boundary and constant scalar curvature is rigid, i.e., it is Einstein or its universal cover is a product of Einstein manifolds (cf. \cite[Theorem 1.3]{Petersen-Chenxu}). Other related results for compact $m$-quasi-Einstein manifold with boundary and constant scalar curvature were discussed in \cite{compact,remarks,He-Petersen-Wylie2012}. Nevertheless, the explicit classification of compact $m$-quasi-Einstein manifolds with boundary and constant scalar curvature is still open. In another direction, Petersen and Wylie \cite{PW} studied rigid gradient Ricci solitons. It is known, by the works of Hamilton \cite{Hamilton}, Ivey \cite{Ivey}, Perelman \cite{Perelman}, Naber \cite{Naber}, Ni-Wallach \cite{NW}, and Cao-Chen-Zhu \cite{CCZ}, that $2$ and $3$-dimensional gradient shrinking Ricci solitons are rigid, and moreover, they are entirely classified. A more recent result due to Cheng and Zhou \cite{ChengZhou}, combined with Fern\'andez-Lop\'ez and Garc\'ia-R\'io \cite{Fl-Gr}, establishes the complete classification of $4$-dimensional gradient shrinking Ricci solitons with constant scalar curvature, which in turn provides a par\-tial solution for a pro\-blem raised by Huai-Dong Cao  (cf. \cite{ChengZhou}). This present work is also motivated by these results on gradient Ricci solitons.

In this article, inspired by the question mentioned earlier and by works due to Cheng and Zhou \cite{ChengZhou}, Fern\'andez-Lop\'ez and Garc\'ia-R\'io \cite{Fl-Gr} and He, Petersen and Wylie \cite{Petersen-Chenxu}, we will establish the complete classification of compact simply connected $3$ and $4$-dimensional $m$-quasi-Einstein manifolds with boundary and constant scalar curvature. To that end, in the same spirit of \cite{Fl-Gr}, we first determine the possible values for the constant scalar curvature of an $n$-dimensional compact $m$-quasi-Einstein manifold with boundary. More precisely, we have the following result.

\begin{theorem}
\label{theo1}
    Let $\big(M^n,\,g,\,u,\,\lambda\big)$ be a nontrivial compact $m$-quasi-Einstein manifold with boundary, $m>1$ and constant scalar curvature $R.$ Then we have:
\begin{eqnarray}
\label{estR}
    R\in\left\{\frac{k(m-n)+n(n-1)}{m+n-k-1}\lambda\,\,; \,k\in\{0,1,\ldots,n-1\}\right\}.
\end{eqnarray}
\end{theorem}

We note that the value of the scalar curvature  in \eqref{estR} can be interpreted in terms of the dimension $k$ of the set of critical points (or, equivalently, of the set of ma\-xi\-mum points); see the proof of Theorem \ref{theo1} in Section \ref{secProofs}. In Example ${\rm (i)},$ we have $R=\frac{n(n-1)\lambda}{m+n-1},$ and the only critical point is the north pole, i.e., $k=0.$ In Exam\-ple ${\rm (ii)},$ $R=(n-1)\lambda,$ and the set of critical points of the potential function $u(t,x)=\sin(\sqrt{\lambda/m}\,t)$ is precisely $\left\{\sqrt{\frac{m}{\lambda}}\frac{\pi}{2}\right\}\times \mathbb{S}^{n-1},$ which has dimension $n-1.$ Finally, in Example ${\rm (iii)},$ $R=\frac{q(m-n)+n(n-1)}{m+n-q-1}\lambda$ and the set of critical points of the potential function is $\{north\,\,pole\}\times \Bbb{S}^q.$

Before discussing our next result, we recall that if an $m$-quasi-Einstein manifold has constant scalar curvature $R$ and $m>1,$ then

	\begin{eqnarray}
\label{eqHj56p}
|\Rc|^2=-\frac{m+n-1}{n(m-1)}(R-n\lambda)\left(R-\frac{n(n-1)}{m+n-1}\lambda\right);
\end{eqnarray} for more details, see \cite[Proposition 3.3]{He-Petersen-Wylie2012} and \cite[Lemma 3.2]{CaseShuWei} (see also Lemma \ref{lemmafund}).

\begin{remark}
\label{remL}
Observe that in considering $R=\frac{n(n-1)}{m+n-1}\lambda$ into (\ref{eqHj56p}), i.e., the lower value of (\ref{estR}), one deduces that $M^n$ is necessarily Einstein. Therefore, it suffices to apply Proposition 2.4 of \cite{Petersen-Chenxu} to conclude that $M^n$ is isometric to the standard hemisphere $\Bbb{S}^n_+.$ Moreover, as we shall see in Proposition \ref{propK11} in Section \ref{secProofs}, there is no compact nontrivial quasi-Einstein manifold with boundary and constant scalar curvature $R=\frac{m+n(n-2)}{m+n-2}\lambda.$
\end{remark}

In the sequel, we shall consider the extremal value case of (\ref{estR}), namely, $R=(n-1)\lambda.$ In this situation, we have the following result which can be compared with \cite[Theorem 1.9]{Petersen-Chenxu}.

\begin{theorem}
\label{theo3}
Let $\big(M^n,\,g,\,u,\,\lambda\big),$ $n\geq 3,$ be a nontrivial simply connected compact $m$-quasi-Einstein manifold with boundary and $m>1.$ Then $M^n$ has constant scalar curvature $R=(n-1)\lambda$ if and only if it is isometric, up to scaling, to the cylinder $I \times N$ with product metric, where $N$ is a compact $\lambda$-Einstein manifold.
\end{theorem}

As a consequence of Theorem \ref{theo1}, Proposition \ref{propK11}, Theorem \ref{theo3} and Proposition 2.4 in  \cite{Petersen-Chenxu}, we shall obtain a classification for compact $3$-dimensional $m$-quasi-Einstein manifolds with boundary and constant scalar curvature. To be precise, we have the following result.

\begin{corollary}\label{theo2}
Let $(M^3,\,g,\,u,\,\lambda)$ be a nontrivial simply connected compact $3$-dimensional $m$-quasi-Einstein manifold with boundary and $m>1.$ Then $M^3$ has constant scalar curvature if and only if it is isometric, up to scaling, to either
    \begin{itemize}
        \item[(i)] the standard hemisphere $\mathbb{S}^{3}_{+}$, or
        \item[(ii)] the cylinder $I\times\mathbb{S}^2$ with the product metric.
    \end{itemize}
\end{corollary}

\vspace{0.10cm}

From now on, we focus on dimension $n=4.$ It is well known that four-dimensional manifolds exhibit fascinating and distinctive geometric features. This is largely due to the fact that, on a four-dimensional oriented compact Riemannian manifold, the bundle of $2$-forms admits an invariant decomposition as a direct sum. We refer the reader to \cite{GS1999} for further details on this specific dimension. In this scenario, one deduces from Theorem \ref{theo1} and Proposition \ref{propK11} that the possible values for the constant scalar curvature $R$ are 
\begin{equation*}
\left\{\frac{12}{m+3}\lambda,\,2\frac{(m+2)}{(m+1)}\lambda,\,3\lambda\right\}.
\end{equation*} 
If $R=\frac{12}{m+3}\lambda,$ it then follows from Remark \ref{remL} that $M^4$ is isometric, up to scaling, to the standard hemisphere $\Bbb{S}^4_+.$ In the case $R=3\lambda,$ it suffices to apply Theorem \ref{theo3} to conclude that $M^4$ is isometric to the cylinder $I\times \Bbb{S}^3$ with product metric. This fact has left open the question of whether $\mathbb{S}^{2}_{+}\times\mathbb{S}^2$ is the unique $4$-dimensional compact quasi-Einstein manifold with boundary and constant scalar curvature $R=2\frac{(m+2)}{(m+1)}\lambda.$ We answer this question:

\begin{theorem}\label{theodim4}
    Let $(M^4,\,g,\,u,\,\lambda)$ be a nontrivial simply connected compact $4$-di\-men\-sional $m$-quasi-Einstein manifold with boundary and $m>1.$ Then $M^4$ has constant scalar curvature $R=2\frac{(m+2)}{(m+1)}\lambda$ if and only if it is isometric, up to scaling, to the product space $\mathbb{S}^{2}_{+}\times\mathbb{S}^2$ with the product metric. 
\end{theorem}

The proof of Theorem \ref{theodim4} is essentially inspired by the work of Cheng and Zhou \cite{ChengZhou}. As a consequence of Theorem \ref{theo1}, Remark \ref{remL}, Theorem \ref{theo3} and Theorem \ref{theodim4}, we get the following classification result. 

\begin{corollary}
\label{corA}
Let $(M^4,\,g,\,u,\,\lambda)$ be a nontrivial simply connected compact $4$-di\-men\-sio\-nal $m$-quasi-Einstein manifold with boundary and $m>1.$ Then $M^4$ has constant scalar curvature if and only if it is isometric, up to scaling, to either
\begin{itemize}
\item[(i)] the standard hemisphere $\mathbb{S}^{4}_{+},$ or
\item[(ii)] the cylinder $I\times\mathbb{S}^3$ with the product metric, or
\item[(iii)] the product space $\mathbb{S}^{2}_{+}\times\mathbb{S}^2$ with the product metric.
\end{itemize}
\end{corollary}

\vspace{0.30cm}
\medskip

In order to prove Theorem \ref{theo1}, we adapt an argument from \cite{Petersen-Chenxu} to determine the possible values of the constant scalar curvature for a compact quasi-Einstein manifold with boundary. Based on this list, we analyze the case $R=(n-1)\lambda,$ which is treated in Theorem \ref{theo3} and corresponds to the cylindrical model $I\times N,$ where $N$ is a compact $\lambda$-Einstein manifold. At this stage, Corollary \ref{coro-tensort} plays a key role, providing a characterization of quasi-Einstein manifolds with constant scalar curvature and vanishing tensor $T,$ defined in Lemma \ref{l1}. 

We then prove Corollary \ref{theo2}, which gives the classification of compact $3$-dimensional quasi-Einstein manifolds with boundary and constant scalar curvature. The proof combines Theorem \ref{theo1}, Proposition \ref{propK11}, Theorem \ref{theo3} and \cite[Proposition 2.4]{Petersen-Chenxu}. In particular, the isoparametric property of $u$ is crucial for establishing Proposition \ref{propK11} and for studying the set of maximum points of $u,$ denoted by $MAX(u).$

The proof of Theorem \ref{theodim4} is considerably more involved. We first consider the tensor $P=Ric-\rho g,$ where $\rho=\frac{(n-1)\lambda -R}{m-1}.$ Diagonalizing $P,$ we will see that, in our setting, the first eigenvalue $\mu_{1}$ of $P$ is zero and therefore, $$Tr(P)=\mu_{2}+\mu_{3}+\mu_{4}\,\,\,\,\hbox{and}\,\,\,\,|P|^2 = \mu_{2}^2 +\mu_{3}^2 +\mu_{4}^2.$$ Furthermore, we will deduce that $$Tr(P)=2m\rho\,\,\,\,\hbox{and}\,\,\,\,|P|^2=2m^2 \rho^2 =\frac{1}{2}(Tr(P))^2.$$ Thus, we need to obtain a third equation involving $P$ in order to determine its eigenvalues. To this end, we derive a formula for $u\Delta (Ric)$ (Lemma \ref{lemKA1}) and then use this to establish a formula for $u\Delta (Tr\,(P^3)),$ where $Tr(P^3)=P_{ij}P_{jl}P_{li}$ (Proposition \ref{propKK1}). Next, we prove Lemma \ref{lemKAp}, which provides an inequality for the operator $L_{m+2}$ acting on the nonnegative function $$|\nabla u|^2 (Tr(P^3)-2m^{3}\rho^{3}).$$ It should be emphasized that obtaining such a suitable nonnegative function involves intricate computations (Propositions \ref{propKK1} and \ref{propKK2}). Again, the isoparametric property of $u$ is essential, as it is used to compare the curvature of $M^4$ with that of the level sets of $u$ via the Gauss equation. These level sets are three-dimensional, and therefore their curvature can be expressed in terms of the Ricci tensor. We then apply an integration by parts argument to show that $$Tr(P^3)-2m^{3}\rho^{3}=0,$$ which implies that the eigenvalues of the Ricci tensor $\lambda_{i},$ $1\leq i\leq 4,$ of $M^4$ are $$\lambda_{1}=\lambda_{2}=\frac{\lambda}{m+1}\,\,\,\hbox{and}\,\,\,\lambda_{3}=\lambda_{4}=\lambda.$$ From this, one concludes that $M^4$ is isometric to $\mathbb{S}^{2}_{+}\times\mathbb{S}^2$ with the product metric.

To prove Corollary \ref{corA}, it suffices to combine Theorem \ref{theo1}, Proposition 2.4 of \cite{Petersen-Chenxu}, Proposition \ref{propK11}, Theorem \ref{theo3} and Theorem \ref{theodim4}.

\vspace{0.30cm}
	\medskip

The rest of this paper is organized as follows. In Section \ref{preliminaries}, we review some basic facts and useful results on $m$-quasi-Einstein manifolds that will be used in the proofs of the main theorems. Some novel lemmas will be discussed in Section \ref{secKL}. Section \ref{secProofs} collects the proofs of Theorem \ref{theo1}, Theorem \ref{theo3} and Corollary \ref{theo2}. Finally, the proofs of Theorem \ref{theodim4} and Corollary \ref{corA} are presented in Section \ref{secthm4}. We also include an appendix to explain a few remarks used throughout this article.

\section{Preliminaries}
\label{preliminaries}

In this section, we review some basic facts and present some features that will play a
fundamental role in the proof of the main results.

\subsection{Background} 
\label{Sec2}

Throughout this paper, we adopt the following convention for the curvatures:

$$\text{Rm}(X,Y)=\nabla^2_{Y,X}-\nabla^2_{X,Y}, \quad  Rm(X,Y,Z,W)=g(\text{Rm}(X,Y)Z,W),$$
$$ K(e_i,e_j)=Rm(e_i,e_j,e_i,e_j), \quad  Ric(X,Y)=\text{tr}\, Rm(X, \cdot, Y, \cdot),$$
$$  R_{ij}=Ric(e_i, e_j), \quad  R=\text{tr}\,Ric.$$  
\vskip 0.2cm

Given a warped product manifold $(I\times_{\varphi} N,\,g=dt^2+\varphi^2(t)g_{_N}),$ where $\varphi$ is a positive smooth (warping) function defined on the interval $I,$ and given a smooth function $f(t,x)=f(t),$ one easily verifies that

\begin{eqnarray}\label{Hlie}
    2\nabla^2 f=2f''(t)dt^2+2f'(t)\varphi(t)\varphi'(t)g_{_N}.
\end{eqnarray}

We also recall the following formulae for warped product manifolds (cf. \cite{kk,O'Neil}).

\begin{proposition}[\cite{O'Neil}]
\label{laux}
The Ricci curvature of a warped product manifold $M=B\times_\varphi F,$ with $l=dim(F),$ must satisfy:
	\begin{itemize}
		\item[(i)] $Ric(X,Y)=Ric_B(X,Y)-\frac{l}{\varphi}\nabla^2\varphi(X,Y),$
		\item[(ii)] $Ric(X,V)=0,$
		\item[(iii)] $Ric(V,W)=Ric_F(V,W)-(\varphi\Delta \varphi+(l-1)|\nabla \varphi|^2)g_F(V,W),$
	\end{itemize} where  $X,\,Y$ and $V,\,W$ are horizontal and vertical vectors, respectively.
\end{proposition}

As a consequence of Proposition \ref{laux}, as observed in \cite{Besse}, if $(M^n,\,g)$ is a warped product manifold with $g= dt^2+\varphi^2(t) g_{_N},$ where $g_{_N}$ is a $\kappa$-Einstein metric with $\kappa>0,$ and if either $\varphi(t)=\alpha t$ or $\varphi(t)=a \sinh(\sqrt{\beta}t)+ b\cosh(\sqrt{\beta}t)$, where $\alpha$ and $\beta$ are positive constants and $a,\,b\in\mathbb{R},$ then the scalar curvature $R$ of $M^n$ cannot be a positive constant.

	\medskip

\subsection{Quasi-Einstein Manifolds}
\label{Sec2.2}

In this subsection, we recall basic facts on $m$-quasi-Einstein manifolds. First of all, we remember that the fundamental equation of an $m$-quasi-Einstein manifold $(M^n,\,g,\,u,\lambda),$ possibly with boundary, is given by
\begin{eqnarray}\label{mqE}
    \nabla^2 u=\frac{u}{m}(Ric-\lambda g),
\end{eqnarray} where $u>0$ in the interior of $M$ and $u=0$ on the boundary $\partial M.$

By tracing \eqref{mqE}, one sees that
\begin{eqnarray}\label{eq-laplacian}
    \Delta u=\frac{u}{m}( R-n\lambda).
\end{eqnarray} This implies that $\Delta u=0$ along $\partial M.$ Besides, Propositions 2.2 and 2.3 of \cite{He-Petersen-Wylie2012} guarantee that $|\nabla u|$ does not vanish on the boundary and it is constant on each component of $\partial M.$ From this, we infer that $\nu=-\frac{\nabla u}{|\nabla u|}$ is the unit outward normal vector field over $\partial M.$ In particular, by the Stokes' formula, $\Delta u$ is not identically zero. Actually, we have
\begin{eqnarray}\label{int laplacian}
\int_M \Delta u \,dM_g &=&\int_{\partial M}\langle \nabla u,\nu\rangle\, dS_g\nonumber\\ &=&\sum_{i=1}^{l}\int_{\partial M_{i}}\langle \nabla u,\nu\rangle\, dS_g =-\sum_{i=1}^{l}|\nabla u|_{\mid_{\partial M_{i}}}\,|\partial M_{i}|< 0,
\end{eqnarray} where $\partial M=\cup_{i=1}^{l} \partial M_{i}$ and $\partial M_{i}$ are the connected components of $\partial M.$

\begin{remark}
\label{remJK}
It follows from (\ref{eq-laplacian}) and (\ref{int laplacian}) that if the scalar curvature $R$ is constant, then $ R< n\lambda$ (see \cite[Corollary 4.3]{He-Petersen-Wylie2012}). Thus, the scalar curvature $R$ cannot be $n\lambda$ in Theorem \ref{theo1}.  
\end{remark}

From now on, we consider an orthonormal frame $\{e_i\}_{i=1}^{n}$ with $e_1=\nu=-\frac{\nabla u}{|\nabla u|}.$ Under these coordinates, since $u=0$ on $\partial M,$ the second fundamental form satisfies
\begin{eqnarray*}
    h_{ab}=-\left\langle\nabla_{e_a}\nu,e_b \right\rangle=\frac{1}{|\nabla u|}\nabla_a\nabla_b u=0,
\end{eqnarray*} for any $2\leq a,b,c,d\leq n.$ Hence, $\partial M$ is totally geodesic. Also, by the Gauss equation, i.e., 
\begin{eqnarray*}
 R_{abcd}^{\partial M}= R_{abcd}-h_{ad}h_{bc}+h_{ac}h_{bd},
\end{eqnarray*} one obtains that
\begin{equation}
\label{eqe1}
 R^{\partial M}= R- 2R_{11}.
\end{equation}

We further recall some important features of $m$-quasi-Einstein manifolds (cf. \cite[Lemma 3.2 and Theorem 2.2]{CaseShuWei}, \cite[Proposition 2.1]{He-Petersen-Wylie2012} and \cite[Lemma 2.1]{compact}).

\begin{lemma}
\label{lemmafund}
Let $(M^n,\,g,\,u,\,\lambda)$ be an $m$-quasi-Einstein manifold with $m>1.$ Then we have:
\begin{enumerate}
\item $\frac{1}{2}u\nabla R=-(m-1)Ric(\nabla u)-(R-(n-1)\lambda)\nabla u;$

\item $\frac{u^2}{m}\left(R-\lambda n\right)+(m-1)|\nabla u|^2 =-\lambda u^2 +\mu,$ where $\mu$ is a constant;
\item \begin{eqnarray*}
\frac{1}{2}\Delta R &=&-\frac{m+2}{2u}\langle \nabla u,\,\nabla R\rangle -\frac{m-1}{m}\left|Ric-\frac{R}{n}g\right|^2\nonumber\\&& -\frac{(n+m-1)}{mn}\left(R-n\lambda\right)\left(R-\frac{n(n-1)}{n+m-1}\lambda\right);
\end{eqnarray*}
\item $u\left(\nabla_{i}R_{jk}-\nabla_{j}R_{ik}\right)=m R_{ijkl}\nabla_{l}u+\lambda\left(\nabla_{i}ug_{jk}-\nabla_{j}u g_{ik}\right)-\left(\nabla_{i}u R_{jk}-\nabla_{j}u R_{ik}\right).$

\end{enumerate}
\end{lemma}

We highlight that Eq. (2) of Lemma \ref{lemmafund} determines a type of ``integrability condition". Besides, Eq. (4) of Lemma \ref{lemmafund} was observed in \cite[Lemma 2.1]{compact}, see also \cite[Proposition 6.2]{He-Petersen-Wylie2012}.

From assertion (1) of Lemma \ref{lemmafund}, if an $m$-quasi-Einstein manifold $M^n$ has constant scalar curvature and $m\neq 1,$ then
\begin{eqnarray}\label{eq-ridu}
    Ric(\nabla u)=\frac{(n-1)\lambda- R}{m-1}\nabla u.
\end{eqnarray} Consequently, the traceless Ricci tensor $\mathring{Ric}$ must satisfy 
\begin{eqnarray}\label{ritdu}
    \mathring{Ric}(\nabla u)=\frac{n(n-1)\lambda-(m+n-1)R}{n(m-1)}\nabla u.
\end{eqnarray}

We now set the covariant $2$-tensor $P$ by 
\begin{eqnarray}\label{tensorP}
    P=Ric-\frac{(n-1)\lambda-R}{m-1}g. 
\end{eqnarray} In this perspective, by assuming that $M$ has constant scalar curvature, we have from \eqref{eq-ridu} that $P(\nabla u)=0.$ Furthermore, by using the orthonormal frame $\{e_i\}_{i=1}^{n}$ that diagonalizes the Ricci tensor, one observes that $P(e_i)=\mu_i e_i.$ In \cite{He-Petersen-Wylie2012}, it was introduced the $4$-tensor $Q$ related to $P$ as follows 

\begin{eqnarray}
\label{tensorQ}
    Q=Rm+\frac{1}{m}P\odot g+\frac{(n-m)\lambda-R}{2m(m-1)}g\odot g,
\end{eqnarray} where $\odot$ stands for the Kulkarni-Nomizu product and $Rm$ is the Riemann tensor. For covariant 2-tensors $S$ and $T,$ the Kulkarni-Nomizu product\footnote{Our definition of Kulkarni-Nomizu product differs from \cite{He-Petersen-Wylie2012} by a constant $1/2$ and sign.} is given by 
\begin{eqnarray}
\label{kulkarniP}
    (S\odot T)_{ijkl}=S_{ik}T_{jl}+S_{jl}T_{ik}-S_{il}T_{jk}-S_{jk}T_{il}.
\end{eqnarray}

With these tools, we have following proposition from \cite[Proposition 6.2]{He-Petersen-Wylie2012}. 

\begin{proposition}
\label{propKL1}
Let $(M^n,\,g,\,u,\,\lambda)$ be an $m$-quasi-Einstein manifold. Then we have:
\begin{equation*}
u(\nabla_{i}P_{jk}-\nabla_{j}P_{ik})= mQ_{ijkl}\nabla_{l}u+\frac{1}{2}(g\odot g)_{ijkl}P_{sl}\nabla_{s}u.
\end{equation*}

\end{proposition}

As a consequence of Proposition \ref{propKL1}, we obtain the following identities, originally establi\-shed by He, Petersen, and Wylie in \cite[Proposition 3.7]{Petersen-Chenxu}. Note that our convention for the Kulkarni-Nomizu product (\ref{kulkarniP}) and $Ric(X,Y)=\text{tr}\, Rm(X, \cdot, Y, \cdot)$ differ from \cite{Petersen-Chenxu}.

\begin{proposition}[\cite{Petersen-Chenxu}]
\label{prop-p}
Let $(M^n,\,g,\,u,\,\lambda)$ be an $m$-quasi-Einstein manifold with constant scalar curvature and $m>1.$ Then we have:
   \begin{enumerate}
   \item $$\frac{u}{m}(\nabla_i P_{jk}-\nabla_j P_{ik})=\frac{u}{m}(\nabla_i R_{jk}-\nabla_j R_{ik})=Q_{ijkl}\nabla_l u,$$
   \item $$ \frac{u}{m}\nabla_i P_{jk}\nabla_i u =\left(\frac{u}{m}\right)^2\left(\left(\lambda-\rho\right)P_{jk}-P_{ik}P_{ij}\right)+Q_{ijkl}\nabla_l u\nabla_i u,$$
   \end{enumerate} where $\rho=\frac{(n-1)\lambda-R}{m-1}.$
\end{proposition}

Now, it is convenient to recall the following terminology (see \cite{Petersen-Chenxu}).

\begin{definition}
An $m$-quasi-Einstein manifold $(M^n,\,g,\,u,\,\lambda)$ is said to be \textit{rigid} if it is Einstein or its universal cover is a product of Einstein manifolds.  
\end{definition}

In \cite{Petersen-Chenxu}, it was established the following result for rigid $m$-quasi-Einstein manifolds.

\begin{proposition}[\cite{Petersen-Chenxu}]
\label{prop-hpw}
A non-trivial complete rigid $m$-quasi-Einstein manifold $(M^n,\,g,\,u,\,\lambda)$ is one of the examples in Table 2.1 of \cite{Petersen-Chenxu}, or its universal cover splits off as $$\widetilde{M}=(M_1,g_1)\times(M_2,g_2)\,\,\,\,\,\,\hbox{with}\,\,\,\,\,\,\,\,u(x,y)=u(y),$$ where $(M_{1},g_1,\,\lambda)$ is a trivial quasi-Einstein manifold and $(M_{2},\,g_2,\,u)$ is one of the examples in Table 2.1 in \cite{Petersen-Chenxu}.  
\end{proposition}

Notice that the example on the hemisphere $\Bbb{S}^n_{+},$ already mentioned in the Introduction, also appear in Table 2.1 of \cite{Petersen-Chenxu}. 

\begin{remark}
\label{remarkpifinite}
It is known that the universal covering of a quasi-Einstein mani\-fold with $\lambda>0$ is compact (including the case with nonempty boundary), and hence its fundamental group $\pi_{1}(M)$ is finite. The proof of this fact is similar to the arguments found in \cite{Fl-Gr2,WylieF}, and can be carried out by combining the techniques used in the proof of  \cite[Theorem 4.1]{He-Petersen-Wylie2012} (see also \cite{Qian}) and \cite[Remark 6.9]{Rimoldi2}.
\end{remark}

Before proceeding, we recall that a non-constant function $f:M\rightarrow\mathbb{R}$ of class at least $C^2$ is said to be \textit{transnormal} if
\begin{eqnarray}\label{eqt}
    |\nabla f|^2=b(f)
\end{eqnarray} for some $C^2$ function $b$ on the range of $f$ in $\Bbb{R}.$ We say that $f$ is \textit{isoparametric} if it is transnormal and there is a continuous function $a$ on the range of $f$ in $\mathbb{R}$ such that
\begin{eqnarray}\label{eqi}
    \Delta f= a(f).
\end{eqnarray} In particular, \eqref{eqt} implies that the level set hypersurfaces of $f$ (i.e., $M_{t}=f^{-1}(t),$ where $t$ is a regular value of $f$) are parallel, and the integral curves of $\nabla f$ are the shortest geodesics connecting the level sets. Besides, \eqref{eqi} guarantees that such hypersurfaces have constant mean curvatures. The preimage of the maximum (respectively, minimum) of an isoparametric (or transnormal) function $f$ is called the {\it focal variety} of $f.$ We refer the reader to \cite{GT,GT2,Miyaoka,Wang} for more details.

By considering that $(M^n,\,g,\,u,\,\lambda)$ is an $m$-quasi-Einstein manifold with constant scalar curvature, one deduces from assertion (2) of Lemma \ref{lemmafund}, for $m>1,$ that 

\begin{eqnarray}
\label{transnormal}
    |\nabla u|^2=\frac{\mu}{m-1}-\frac{ R+(m-n)\lambda}{m(m-1)}u^2.
\end{eqnarray} Consequently, the potential function $u$ is transnormal, namely,
\begin{eqnarray}\label{eqb}
  b(u)=\frac{\mu}{m-1}-\frac{ R+(m-n)\lambda}{m(m-1)}u^2.
\end{eqnarray} Therefore, it follows from \eqref{eq-laplacian} that the potential function $u$ is isoparametric.

Concerning the regularity of the potential function, for an $m$-quasi-Einstein manifold $(M^n,\,g,\,u,\,\lambda),$ it is known that $u$ and $g$ are real analytic in harmonic coordinates (cf. Proposition 2.4 in \cite{He-Petersen-Wylie2012}). In particular, the critical level sets of $u$ have zero measure.

A central object in our approach is the set of maximum points of $u$ given by
\begin{eqnarray*}
    MAX(u)=\{p\in M:\, u(p)=u_{\max}\}.
\end{eqnarray*}

\begin{remark}
\label{remK}
In the compact case with $m>1,$ notice that every point in $MAX(u),$ which clearly is an interior point, must be a critical point. Moreover, from the fact that $u$ is a transnormal function and \eqref{transnormal}, one deduces that the critical points of $u$ have the same value. Therefore, $MAX(u)=Crit(u)$ for nontrivial compact $m$-quasi-Einstein manifolds. 
\end{remark}

\section{Key Lemmas}
\label{secKL}
In this section, we shall provide several novel lemmas that will be used in the proofs of the main results. We start by recalling certain tensors that will be employed in the proofs of Theorem~\ref{theo3} and Corollary~\ref{theo2}. For a Riemannian manifold $(M^n,\,g),\,n\geq 4,$ the Weyl tensor is given by
\begin{eqnarray}\label{weyl}
    W_{ijkl}=R_{ijkl}-\frac{1}{n-2}(A\odot g)_{ijkl}, 
\end{eqnarray} where $A=Ric-\frac{R}{2(n-1)} g$ stands for the Schouten tensor. Another tensor that will be useful in our discussion is the Cotton tensor, for $n\geq 3,$
\begin{equation}\label{cot1}
    C_{ijk}=\nabla_i R_{jk}-\nabla_j R_{ik}-\frac{1}{2(n-1)}(\nabla_i R g_{jk}-\nabla_j R g_{ik}).
\end{equation} Next, for $n\geq 4,$ we have
\begin{equation}\label{cot2}
    C_{ijk}=-\left(\frac{n-2}{n-3}\right)\nabla_l W_{ijkl}.
\end{equation} Notice that $C_{ijk}$ is skew-symmetric in the first two indices and trace-free in any two indices.

It turns out that, on an $m$-quasi-Einstein manifold, we may express the Cotton tensor in terms of the Weyl tensor and an auxiliary 3-tensor $T_{ijk}$ as follows (see \cite[Lemma 2.2]{compact}).

\begin{lemma}[\cite{compact}]
\label{l1}
    Let $(M^n,\,g,\,u,\,\lambda)$ be an $m$-quasi-Einstein manifold. Then it holds
    \begin{equation}\label{eq3}
        uC_{ijk}=mW_{ijkl}\nabla_l u+T_{ijk},
    \end{equation} where the 3-tensor $T_{ijk}$ is given by
    \begin{eqnarray*}
        T_{ijk}&=&\frac{m+n-2}{n-2}(R_{ik}\nabla_j u-R_{jk}\nabla_i u)+\frac{m}{n-2}(R_{jl}\nabla_l ug_{ik}-R_{il}\nabla_l ug_{jk})\\
       & &+\frac{(n-1)(n-2)\lambda+m R}{(n-1)(n-2)}(\nabla_i ug_{jk}-\nabla_j ug_{ik})-\frac{u}{2(n-1)}(\nabla_i R g_{jk}-\nabla_j R g_{ik}).
    \end{eqnarray*}
    \end{lemma}

We highlight that the tensor $T_{ijk}$ has the same symmetries as the Cotton tensor and it is motivated by the approach employed by Cao and Chen in \cite{CC} in their study of Bach-flat gradient Ricci solitons; see also \cite{CH,QY}. Besides, it is convenient to express the tensor $T_{ijk}$ in terms of the traceless Ricci tensor

\begin{eqnarray}
\nonumber
    T_{ijk}&=&\frac{m+n-2}{n-2}(\mathring{R}_{ik}\nabla_j u-\mathring{R}_{jk}\nabla_i u)+\frac{m}{n-2}(\mathring{R}_{jl}\nabla_l u g_{ik}-\mathring{R}_{il}\nabla_l u g_{jk})\\\nonumber
    & &+\frac{n(n-1)\lambda-(m+n-1)R}{n(n-1)}(\nabla_i u g_{jk}-\nabla_j u g_{ik})\\\label{tensort}
    & &-\frac{u}{2(n-1)}(\nabla_i R g_{jk}-\nabla_j R g_{ik}).
\end{eqnarray}

With aid of this notation, we have the following lemma.

\begin{lemma}\label{l2}
Let $(M^n,g,u,\lambda)$ be an $m$-quasi-Einstein manifold with constant scalar curvature. Then we have:
\begin{eqnarray}
\label{righ}
    \mathring{R}_{ik}T_{ijk}\nabla_j u&=&\frac{m+n-2}{n-2}|\mathring{Ric}|^2|\nabla u|^2-\frac{2m+n-2}{n-2}\mathring{Ric}^2(\nabla u,\nabla u)\nonumber\\
    & &+\frac{\left(n(n-1)\lambda-(m+n-1)R\right)^2}{n^2 (n-1)(m-1)}|\nabla u|^2\nonumber\\
    &=& \frac{n-2}{2(m+n-2)}|T|^2,
\end{eqnarray} where $\mathring{Ric}^{2}_{ij}=\mathring{R}_{ik}\mathring{R}_{kj}$.
\end{lemma}

\begin{proof} 
  By using that the scalar curvature $R$ is constant and Eq. \eqref{tensort}, one obtains that 
\begin{eqnarray*}
    \mathring{R}_{ik}T_{ijk}&=&\frac{m+n-2}{n-2}(|\mathring{Ric}|^2 \nabla_j u-\mathring{R}_{ik}\mathring{R}_{jk}\nabla_i u)-\frac{m}{n-2}\mathring{R}_{ik}\mathring{R}_{il}\nabla_l u g_{jk}\\
    & &+\frac{n(n-1)\lambda-(m+n-1)R}{n(n-1)}\mathring{R}_{ik}\nabla_i u g_{jk}\\
    &=&\frac{m+n-2}{n-2}(|\mathring{Ric}|^2 \nabla_j u-\mathring{R}_{ik}\mathring{R}_{jk}\nabla_i u)-\frac{m}{n-2}\mathring{R}_{ij}\mathring{R}_{il}\nabla_l u\\
    & &+\frac{n(n-1)\lambda-(m+n-1)R}{n(n-1)}\mathring{R}_{ij}\nabla_i u.
\end{eqnarray*} Applying this for $\nabla_j u,$ we see that
\begin{eqnarray*}
    \mathring{R}_{ik}T_{ijk}\nabla_j u&=&\frac{m+n-2}{n-2}|\mathring{Ric}|^2|\nabla u|^2-\frac{m+n-2}{n-2}\nabla_j u\mathring{R}_{ik}\mathring{R}_{jk}\nabla_i u\\
    & &-\frac{m}{n-2}\nabla_j u\mathring{R}_{ij}\mathring{R}_{il}\nabla_l u+\frac{n(n-1)\lambda-(m+n-1)R}{n(n-1)}\mathring{Ric}(\nabla u,\nabla u)\\
    &=&\frac{m+n-2}{n-2}|\mathring{Ric}|^2|\nabla u|^2-\frac{2m+n-2}{n-2}\mathring{Ric}^2(\nabla u,\nabla u)\\
    & &+\frac{n(n-1)\lambda-(m+n-1)R}{n(n-1)}\mathring{Ric}(\nabla u,\nabla u).
\end{eqnarray*} So, it suffices to use \eqref{ritdu} in the last term of the above equality in order to infer the first equality in \eqref{righ}. 

Finally, since $T$ is trace-free in any two indices and skew-symmetric in their first two indices, we get

\begin{eqnarray*}
    \mathring{R}_{ik}T_{ijk}\nabla_j u&=&\frac{1}{2}(\mathring{R}_{ik}T_{ijk}\nabla_j u-\mathring{R}_{ik}T_{jik}\nabla_j u)\\
    &=&\frac{1}{2}T_{ijk}(\mathring{R}_{ik}\nabla_j u-\mathring{R}_{jk}\nabla_i u)\\
    &=&\frac{n-2}{2(m+n-2)}|T|^2,
\end{eqnarray*} where in the last equality we have used \eqref{tensort}. This finishes the proof of the lemma. 
\end{proof}

As a consequence of Lemma \ref{l2}, by considering the aforementioned orthonormal frame $\{e_{i}\}_{i=1}^{n}$ with $e_1=-\frac{\nabla u}{|\nabla u|}$ so that $\mathring{Ric}(e_i)=\xi_i e_i,$ we obtain the following result.

\begin{corollary}
\label{coro-tensort}
Let $(M^n,\,g,\,u,\,\lambda)$ be an $m$-quasi-Einstein manifold with constant scalar curvature and $m>1.$ Then $T$ is identically zero if and only if the Ricci tensor has at most two different eigenvalues, one of them has multiplicity at least $n-1$ and its eigenspace corresponds to the orthogonal complement of $\nabla u.$
\end{corollary}

\begin{proof}
Taking into account that $\xi_1=\frac{n(n-1)\lambda-(m+n-1)R}{n(m-1)},$ one deduces from \eqref{righ} that 
\begin{eqnarray*}
    \frac{n-2}{2(m+n-2)}|T|^2&=&\left[\frac{m+n-2}{n-2}\sum_{i=1}^{n}\xi_{i}^{2}+\frac{m-1}{n-1}\xi_{1}^{2}\right]|\nabla u|^2-\frac{2m+n-2}{n-2}\xi_{1}^{2}|\nabla u|^2\\
    &=&\frac{m+n-2}{n-2}\left[\sum_{i=2}^{n}\xi_{i}^{2}-\frac{1}{n-1}\xi_{1}^{2}\right]|\nabla u|^2
\end{eqnarray*} on the regular points of the potential function $u.$ Moreover, since $Tr(\mathring{Ric})=\sum_{i=1}^{n}\xi_i=0,$ we infer
\begin{eqnarray*}
    \frac{n-2}{2(m+n-2)}|T|^2&=&\frac{m+n-2}{n-2}\left[\sum_{i=2}^{n}\xi_{i}^{2}-\frac{1}{n-1}\left(\sum_{i=2}^{n}\xi_i\right)^2\right]|\nabla u|^2.
\end{eqnarray*} By the Cauchy-Schwarz inequality, we conclude that $T\equiv 0$ if and only if the Ricci tensor has at most two different eigenvalues with $\lambda_2=\ldots=\lambda_n$ at regular points of $u$, for eigenvalues of the Ricci given by $\lambda_i=\xi_i+\frac{R}{n}$. To conclude the proof, it suffices to recall that $u$ is real analytical in harmonic coordinates and consequently, the set of critical points of $u$ has zero measure in $M.$
\end{proof}

In the remainder of this section, we establish several key lemmas, valid in arbitrary dimension $n\geq 3,$ which will play a crucial role in the proof of Theorem \ref{theodim4}. Our goal is to derive an explicit expression for $u\Delta \left(Tr(P^3)\right)$ (see Lemma \ref{lemKA3}). To this end, we first compute a formula for $u\left(\Delta Ric\right).$
 
 \begin{lemma}
 \label{lemKA1}
Let $(M^n,\,g)$ be an $n$-dimensional Riemannian manifold satisfying (\ref{mqE}). Then we have:
\begin{eqnarray*}
u\left(\Delta R_{ik}\right)&=& \nabla_{i}R_{sk}\nabla_{s}u + m\nabla_{k}R_{is}\nabla_{s}u+\frac{u}{2}\nabla_{i}\nabla_{k}R+\frac{1}{2}\nabla_{i}u\nabla_{k}R\nonumber\\&&+ \frac{(m+1)}{m}uR_{is}R_{sk}+2uR_{jiks}R_{js}-(m+2)\nabla_{s}R_{ik}\nabla_{s}u\nonumber\\&&-\frac{u}{m}\left(R-(m+n-2)\lambda\right)R_{ik}+\frac{\lambda u}{m}\left(R-(n-1)\lambda\right)g_{ik}.
\end{eqnarray*}
 \end{lemma}
 \begin{proof}
Firstly, it follows from assertion (4) of Lemma \ref{lemmafund} that 
 \begin{equation*}
 u\nabla_{j}R_{ik}=u\nabla_{i}R_{jk}+mR_{jikl}\nabla_{l}u + \lambda\left(\nabla_{j}u g_{ik}-\nabla_{i}u g_{jk}\right)-\left(\nabla_{j}u R_{ik}-\nabla_{i}u R_{jk}\right).
 \end{equation*} This jointly with the fact that $\nabla_{j}\left(u\nabla_{j}R_{ik}\right)=\nabla_{j}u \nabla_{j}R_{ik}+u\Delta R_{ik}$ gives
 
 \begin{eqnarray*}
 u\Delta R_{ik} &=& \nabla_{j}\left(u\nabla_{j}R_{ik}\right)-\nabla_{j}u \nabla_{j}R_{ik}\nonumber\\&=& \nabla_{j}\left(u\nabla_{i}R_{jk}+mR_{jikl}\nabla_{l}u + \lambda\left(\nabla_{j}u g_{ik}-\nabla_{i}u g_{jk}\right)-\left(\nabla_{j}u R_{ik}-\nabla_{i}u R_{jk}\right)\right)\nonumber\\&&-\nabla_{j}u \nabla_{j}R_{ik}\nonumber\\&=& 
 \nabla_{j}u \nabla_{i}R_{jk}+u\nabla_{j}\nabla_{i}R_{jk}+m\nabla_{j}R_{jikl}\nabla_{l}u+mR_{jikl}\nabla_{j}\nabla_{l}u+\lambda \Delta u g_{ik}\nonumber\\&&-\lambda \nabla_{k}\nabla_{i}u-\Delta u R_{ik}-\nabla_{j}u\nabla_{j}R_{ik} +\nabla_{j}\nabla_{i}u R_{jk}+\nabla_{i}u \nabla_{j}R_{jk}-\nabla_{j}u\nabla_{j}R_{ik}.
 \end{eqnarray*} Next, by using the twice contracted second Bianchi identity $(\nabla_{j}R_{jk}=\frac{1}{2}\nabla_{k}R)$ and the first contracted second Bianchi identity
  $(\nabla_{j}R_{jikl}=\nabla_{k}R_{il}-\nabla_{l}R_{ik}),$ one sees that
  
  \begin{eqnarray}
  \label{lkm19}
   u\Delta R_{ik} &=& -\nabla_{j}u\nabla_{j}R_{ik} + \nabla_{j}u \nabla_{i}R_{jk} + u \nabla_{j}\nabla_{i}R_{jk}+m\left(\nabla_{k}R_{il}-\nabla_{l}R_{ik}\right)\nabla_{l}u \nonumber\\&&+ mR_{jikl}\nabla_{j}\nabla_{l}u+\lambda \Delta u g_{ik}-\lambda \nabla_{k}\nabla_{i}u-\Delta u R_{ik} -\nabla_{j}u\nabla_{j}R_{ik}\nonumber\\&&+\nabla_{j}\nabla_{i}u R_{jk}+\frac{1}{2}\nabla_{i}u\nabla_{k}R\nonumber\\&=&  -\nabla_{j}u\nabla_{j}R_{ik} + \nabla_{j}u \nabla_{i}R_{jk} +\frac{u}{2}\nabla_{i}\nabla_{k}R + uR_{is}R_{sk}+uR_{jiks}R_{js}\nonumber\\&&+m\left(\nabla_{k}R_{il}-\nabla_{l}R_{ik}\right)\nabla_{l}u + mR_{jikl}\nabla_{j}\nabla_{l}u+\lambda \Delta u g_{ik}-\lambda \nabla_{k}\nabla_{i} u\nonumber\\&&- \Delta u R_{ik} -\nabla_{j}u\nabla_{j}R_{ik} + \nabla_{j}\nabla_{i}u R_{jk} + \frac{1}{2}\nabla_{i}u\nabla_{k}R,
  \end{eqnarray} where  in the last equality we have used the Ricci identity, i.e., $$\nabla_{j}\nabla_{i}R_{jk}=\nabla_{i}\nabla_{j}R_{jk}+R_{jijs}R_{sk}+R_{jiks}R_{js}.$$ Plugging  (\ref{mqE}) and (\ref{eq-laplacian}) into (\ref{lkm19}) yields
  
  \begin{eqnarray*}
   u\Delta R_{ik} &=& -\nabla_{j}u\nabla_{j}R_{ik} + \nabla_{j}u \nabla_{i}R_{jk} +\frac{u}{2}\nabla_{i}\nabla_{k}R + uR_{is}R_{sk}+uR_{jiks}R_{js}\nonumber\\&&+m\left(\nabla_{k}R_{il}-\nabla_{l}R_{ik}\right)\nabla_{l}u + u R_{jikl}\left(R_{jl}-\lambda g_{jl}\right) + \frac{\lambda u}{m}\left(R-\lambda n\right)g_{ik}\nonumber\\&&-\frac{\lambda u}{m}\left(R_{ki}-\lambda g_{ki}\right)-\frac{u}{m}\left(R-\lambda n\right)R_{ik}-\nabla_{j}u\nabla_{j}R_{ik} \nonumber\\&&+ \frac{u}{m}\left(R_{ji}-\lambda g_{ji}\right)R_{jk}+\frac{1}{2}\nabla_{i}u\nabla_{k}R\nonumber\\&=& \nabla_{i}R_{jk}\nabla_{j}u + m\nabla_{k}R_{il}\nabla_{l}u + \frac{u}{2}\nabla_{i}\nabla_{k}R + \frac{1}{2}\nabla_{i}u\nabla_{k}R + \frac{(m+1)}{m}uR_{is}R_{sk}\nonumber\\&&+2 uR_{jiks}R_{js}-(m+2)\nabla_{j}R_{ik}\nabla_{j}u + \left(\lambda u-\frac{\lambda u}{m} -\frac{u}{m}\left(R-\lambda n\right)-\frac{\lambda u}{m}\right)R_{ik}\nonumber\\&& +\frac{\lambda u}{m}\left(R-(n-1)\lambda\right)g_{ik}.
  \end{eqnarray*} Rearranging terms, one concludes that
  
  \begin{eqnarray*}
     u\Delta R_{ik} &=& \nabla_{i}R_{sk}\nabla_{s}u +m\nabla_{k}R_{is}\nabla_{s}u + \frac{u}{2}\nabla_{i}\nabla_{k}R + \frac{1}{2}\nabla_{i}u\nabla_{k}R \nonumber\\&&+\frac{(m+1)}{m}u R_{is}R_{sk} +2u R_{jiks}R_{js} - (m+2)\nabla_{s}R_{ik} \nabla_{s}u\nonumber\\&& - \frac{u}{m}\left(R-(m+n-2)\lambda \right)R_{ik} +\frac{\lambda u}{m}\left(R-(n-1)\lambda \right)g_{ik},
  \end{eqnarray*} as we wanted to prove. 
 \end{proof}

As an application of Lemma \ref{lemKA1}, we are able to obtain an useful expression for $\Delta (Ric^3)_{ik}=\Delta\left(R_{ij}R_{jl}R_{lk}\right).$ 

 \begin{lemma}
 \label{lemKA2}
Let $(M^n,\,g)$ be an $n$-dimensional Riemannian manifold satisfying (\ref{mqE}). Then we have:
\begin{eqnarray*}
u\Delta (Ric^3)_{ik}&+&(m+2)\nabla_{s}u\nabla_{s}(Ric^3)_{ik}\nonumber\\&=& \nabla_{i}R_{sj}\nabla_{s}u R_{jl}R_{lk} + \nabla_{j}R_{sl}\nabla_{s}u R_{ij}R_{lk} + \nabla_{l}R_{sk}\nabla_{s}u R_{ij}R_{jl}\nonumber\\&& 
+2u\left(\nabla_{s}R_{ij}\nabla_{s}R_{jl} R_{lk}+\nabla_{s}R_{ij} R_{jl}\nabla_{s}R_{lk}+R_{ij}\nabla_{s}R_{jl}\nabla_{s}R_{lk}\right)\nonumber\\&& 
+m\left(\nabla_{j}R_{is}\nabla_{s}u R_{jl}R_{lk} + \nabla_{l}R_{js}\nabla_{s}u R_{ij}R_{lk} +\nabla_{k}R_{ls}\nabla_{s}u R_{ij}R_{jl}\right)\nonumber\\&&+\frac{u}{2}\left(\nabla_{i}\nabla_{j}R R_{jl}R_{lk}+\nabla_{j}\nabla_{l}R R_{ij}R_{lk}+\nabla_{l}\nabla_{k}R R_{ij}R_{jl}\right)\nonumber\\&& +\frac{1}{2}\left(\nabla_{i}u\nabla_{j}R R_{jl}R_{lk}+\nabla_{j}u \nabla_{l}R R_{ij}R_{lk}+\nabla_{l}u \nabla_{k}R R_{ij}R_{jl}\right)\nonumber\\&& + \frac{(m+1)}{m} u\left(R_{is}R_{sj}R_{jl}R_{lk}+R_{js}R_{sl}R_{ij}R_{lk}+R_{ls}R_{sk}R_{ij}R_{jl}\right)\nonumber\\&&+2 u \left(R_{dijs}R_{ds}R_{jl}R_{lk}+R_{djls}R_{ds}R_{ij}R_{lk}+ R_{dlks}R_{ds}R_{ij}R_{jl}\right)\nonumber\\&& -3\frac{u}{m}\left(R-(m+n-2)\lambda\right)(Ric^3)_{ik} +3\frac{\lambda u}{m}\left(R-(n-1)\lambda\right)R_{il}R_{lk}. 
\end{eqnarray*}
\end{lemma}

\begin{proof}
One easily verifies that

\begin{eqnarray}
\label{eqkp1za}
u\Delta (Ric^3)_{ik}&=& u\Delta (R_{ij}R_{jl}R_{lk})\nonumber\\&=& (u\Delta R_{ij}) R_{jl}R_{lk} +R_{ij}(u\Delta R_{jl}) R_{lk} + R_{ij}R_{jl}(u\Delta R_{lk})\nonumber\\&&
+2u\left(\nabla_{s}R_{ij}\nabla_{s}R_{jl} R_{lk}+\nabla_{s}R_{ij} R_{jl}\nabla_{s}R_{lk}+R_{ij}\nabla_{s}R_{jl}\nabla_{s}R_{lk}\right).
\end{eqnarray} Next, it follows from Lemma \ref{lemKA1} that

\begin{eqnarray}
\label{kl1aq2A}
u\left(\Delta R_{ij}\right) R_{jl}R_{lk}&=& \nabla_{i}R_{sj}\nabla_{s}u R_{jl}R_{lk}+m\nabla_{j}R_{is}\nabla_{s}u R_{jl}R_{lk}+\frac{u}{2}\nabla_{i}\nabla_{j}R R_{jl}R_{lk}\nonumber\\&& + \frac{1}{2}\nabla_{i}u \nabla_{j}R R_{jl}R_{lk}+ \frac{(m+1)}{m}u R_{is}R_{sj}R_{jl}R_{lk}+ 2u R_{dijs}R_{ds}R_{jl}R_{lk}\nonumber\\&&-(m+2)\nabla_{s}R_{ij}\nabla_{s}u R_{jl}R_{lk} -\frac{u}{m}\left(R-(m+n-2)\lambda\right)R_{ij}R_{jl}R_{lk}\nonumber\\&&+\frac{\lambda u}{m}\left(R-(n-1)\lambda\right)R_{il}R_{lk},
\end{eqnarray}

\begin{eqnarray}
\label{kl1aq2B}
R_{ij}\left(u\Delta R_{jl}\right)R_{lk}&=& \nabla_{j}R_{sl}\nabla_{s}u R_{ij}R_{lk}+m\nabla_{l}R_{js}\nabla_{s}u R_{ij}R_{lk}+\frac{u}{2}\nabla_{j}\nabla_{l}R R_{ij}R_{lk}\nonumber\\&& + \frac{1}{2}\nabla_{j}u \nabla_{l}R R_{ij}R_{lk}+ \frac{(m+1)}{m}u R_{js}R_{sl}R_{ij}R_{lk}+ 2u R_{djls}R_{ds}R_{ij}R_{lk}\nonumber\\&&-(m+2)\nabla_{s}R_{jl}\nabla_{s}u R_{ij}R_{lk} -\frac{u}{m}\left(R-(m+n-2)\lambda\right)R_{jl}R_{ij}R_{lk}\nonumber\\&&+\frac{\lambda u}{m}\left(R-(n-1)\lambda\right)R_{il}R_{lk}
\end{eqnarray} and

\begin{eqnarray}
\label{kl1aq2C}
R_{ij}R_{jl}\left(u\Delta R_{lk}\right)&=& \nabla_{l}R_{sk}\nabla_{s}u R_{ij}R_{jl}+m\nabla_{k}R_{ls}\nabla_{s}u R_{ij}R_{jl}+\frac{u}{2}\nabla_{l}\nabla_{k}R R_{ij}R_{jl}\nonumber\\&& + \frac{1}{2}\nabla_{l}u \nabla_{k}R R_{ij}R_{jl}+ \frac{(m+1)}{m}u R_{ls}R_{sk}R_{ij}R_{jl}+ 2u R_{dlks}R_{ds}R_{ij}R_{jl}\nonumber\\&&-(m+2)\nabla_{s}R_{lk}\nabla_{s}u R_{ij}R_{jl} -\frac{u}{m}\left(R-(m+n-2)\lambda\right)R_{ij}R_{jl}R_{lk}\nonumber\\&&+\frac{\lambda u}{m}\left(R-(n-1)\lambda\right)R_{ij}R_{jk}.
\end{eqnarray} Therefore, inserting (\ref{kl1aq2A}), (\ref{kl1aq2B}) and (\ref{kl1aq2C}) into (\ref{eqkp1za}) yields the asserted result.  

\end{proof}

As a consequence of Lemma \ref{lemKA2}, we deduce the following corollary.

\begin{corollary}
 \label{corKA2}
Let $(M^n,\,g)$ be an $n$-dimensional Riemannian manifold satisfying (\ref{mqE}) with constant scalar curvature. Then we have:
\begin{eqnarray}
    u\Delta \left(Tr(Ric^3)\right)&+&(m+2)\nabla_s u \nabla_s (Tr(Ric^3))\nonumber\\&=&3(m+1)\nabla_i R_{sj}R_{jl}R_{il}\nabla_s u+\frac{3(m+1)u}{m}Ric_{ij}^{2}Ric^{2}_{ij}+6uR_{ds}R_{dijs}R_{jl}R_{il}\nonumber\\
    & &-\frac{3u}{m}\left(R-(m+n-2)\lambda\right)Tr(Ric^3)\nonumber\\
    & &+\frac{3\lambda u}{m}\left(R-(n-1)\lambda\right)|Ric|^2+6u\nabla_s R_{ij} \nabla_s R_{jl}R_{il},\nonumber
\end{eqnarray} where $Tr (Ric^3)=R_{ij}R_{jl}R_{li}$ and $Ric^{2}_{ij}=R_{ik}R_{kj}.$

\end{corollary}

\begin{proof}
By using that $M^n$ has constant scalar curvature into Lemma \ref{lemKA2}, one deduces that
\begin{eqnarray*}
    u\Delta Ric^{3}_{ik}&=&\left(\nabla_i R_{sj} R_{jl}R_{lk}+\nabla_j R_{sl} R_{ij}R_{lk}+\nabla_l R_{sk}  R_{ij}R_{jl}\right)\nabla_s u\\
    & &+m\left(\nabla_j R_{is} R_{jl}R_{lk}+\nabla_l R_{js} R_{ij}R_{lk}+\nabla_k R_{ls} R_{ij}R_{jl}\right)\nabla_s u\\
    & &+\frac{m+1}{m}u\left(Ric^{2}_{ij}R_{jl}R_{lk}+Ric^{2}_{jl}R_{ij}R_{lk}+Ric^{2}_{lk}R_{ij}R_{jl}\right)\\
    & &+2uR_{ds}\left(R_{dijs}R_{jl}R_{lk}+R_{djls}R_{ij}R_{lk}+R_{dlks}R_{ij}R_{jl}\right)\\
    & &-(m+2)\nabla_s\left(R_{ij}R_{jl}R_{lk}\right)\nabla_s u-\frac{3u}{m}[R-(m+n-2)\lambda]R_{ij}R_{jl}R_{lk}\\
    & &+2u(\nabla_s R_{ij}\nabla_s R_{jl}R_{lk}+\nabla_s R_{ij}R_{jl}\nabla_s R_{lk}+R_{ij}\nabla_s R_{jl}\nabla_s R_{lk})\\
    & &+\frac{3\lambda u}{m}[R-(n-1)\lambda]R_{is}R_{sk}.
\end{eqnarray*} Besides, tracing the above expression, one sees that
\begin{eqnarray*}
    u\Delta Tr(Ric^3)&=&\left(\nabla_i R_{sj} R_{jl}R_{li}+\nabla_j R_{sl} R_{ij}R_{li}+\nabla_l R_{si}  R_{ij}R_{jl}\right)\nabla_s u\\
    & &+m[\nabla_j R_{is} R_{jl}R_{li}+\nabla_l R_{js} R_{ij}R_{li}+\nabla_i R_{ls} R_{ij}R_{jl}]\nabla_s u\\
    & &+\frac{m+1}{m}u[Ric^{2}_{ij}R_{jl}R_{li}+Ric^{2}_{jl}R_{ij}R_{li}+Ric^{2}_{li}R_{ij}R_{jl}]\\
    & &+2uR_{ds}[R_{dijs}R_{jl}R_{li}+R_{djls}R_{ij}R_{li}+R_{dlis}R_{ij}R_{jl}]\\
    & &-(m+2)\nabla_s[R_{ij}R_{jl}R_{li}]\nabla_s u-\frac{3u}{m}[R-(m+n-2)\lambda]R_{ij}R_{jl}R_{li}\\
    & &+2u(\nabla_s R_{ij}\nabla_s R_{jl}R_{li}+\nabla_s R_{ij}R_{jl}\nabla_s R_{li}+R_{ij}\nabla_s R_{jl}\nabla_s R_{li})\\
     & &+\frac{3\lambda u}{m}[R-(n-1)\lambda]R_{is}R_{si}\\
    &=&(m+1)[\nabla_i R_{sj}R_{jl}R_{li}+\nabla_j R_{sl}R_{ij}R_{il}+\nabla_l R_{is}R_{ij}R_{jl}]\nabla_s u\\
    & &+\frac{3(m+1)u}{m}Ric_{ij}^{2}Ric_{ij}^{2}+6uR_{ds}R_{dijs}R_{jl}R_{il}\\
    & &-(m+2)\nabla_s(Tr(Ric^3))\nabla_s u-\frac{3u}{m}\left(R-(m+n-2)\lambda\right)Tr(Ric^3)\\
    & &+\frac{3\lambda u}{m}\left(R-(n-1)\lambda\right)|Ric|^2+6u\nabla_s R_{ij}\nabla_s R_{jl}R_{il}.
\end{eqnarray*} The result then follows from the fact that $\nabla_i R_{sj}R_{jl}R_{li}=\nabla_j R_{sl}R_{ij}R_{il}=\nabla_l R_{is}R_{ij}R_{jl}.$
\end{proof}

Proceeding, we derive an expression for $u\Delta \left(Tr(P^3)\right).$ This will serve as the basis for establishing an inequality (see Lemma~\ref{lemKAp}) involving a suitable nonnegative function depending on $Tr(P^3),$ which is essential for the proof of Theorem~\ref{theodim4}.

\begin{lemma}
\label{lemKA3}
Let $(M^n,\,g)$ be an $n$-dimensional Riemannian manifold satisfying (\ref{mqE}) with constant scalar curvature and $m>1.$ Then we have:
    \begin{eqnarray*}
    u\Delta Tr(P^3)&=&3(m+1)\left(\nabla_i P_{sj}P_{jl}P_{il}\nabla_s u+2\rho\nabla_i P_{sj}P_{ij}\nabla_s u\right)\\
    & &+6u\left(\nabla_s P_{ij}\nabla_s P_{jl}P_{il}+\rho\nabla_s P_{ij}\nabla_s P_{ij}\right)\\
    & &+6u\left(P_{ds}R_{dijs}P_{jl}P_{il}+2\rho P_{ds}R_{dijs}P_{ij}\right)-(m+2)\nabla_s(Tr(P^3))\nabla_s u\\
      & &+\frac{3(m+1)u}{m}Tr(P^4)+\frac{3u}{m}\left(3(m+1)\rho+(m-1)\lambda\right)Tr(P^3)\\
    & &+\frac{3\rho u}{m}\left((m+3)\rho+2(m-1)\lambda\right)|P|^2\\
    & &+\frac{3\rho^2 u}{m}\left((m+1)\rho+(m-1)\lambda\right)Tr(P)\\
    & &+6\rho^3 u \left((m+n-1)\rho-(n-1)\lambda\right),
\end{eqnarray*}
\end{lemma}
\begin{proof} Initially, we compute $Ric^{3}_{ik}$ in terms of $P=Ric-\rho g,$ where $\rho=\frac{(n-1)\lambda - R}{m-1}.$ Indeed, we have

\begin{eqnarray*}
    Ric^{3}_{ik}&=&R_{ij}R_{jl}R_{lk}\\
    &=&(P_{ij}+\rho g_{ij})(P_{jl}+\rho g_{jl})(P_{lk}+\rho g_{lk})\\
    &=&P_{ij}P_{jl}P_{lk}+P_{ij}P_{jl}\rho g_{lk}+P_{ij}\rho g_{jl}P_{lk}+P_{ij}\rho g_{jl}\rho g_{lk}\\
    & &+\rho g_{ij}P_{jl}P_{lk}+\rho g_{ij}P_{jl}\rho g_{lk}+\rho g_{ij}\rho g_{jl} P_{lk}+\rho g_{ij}\rho g_{jl}\rho g_{lk}\\
    &=&P_{ik}^{3}+3\rho P_{ik}^{2}+3\rho^2 P_{ik}+\rho^3 g_{ik}.
\end{eqnarray*} Whence, it follows that 

\begin{equation}
\label{trR-trPaa}
    Tr(Ric^3)=Ric^{3}_{ii}=Tr(P^3)+3\rho|P|^2+3\rho^2 Tr(P)+n\rho^3.
\end{equation} Next, notice that 
\begin{eqnarray*}
Tr (P)=\frac{R(m+n-1)-n(n-1)\lambda}{m-1}
\end{eqnarray*} and moreover, by Proposition 3.3 in \cite{Petersen-Chenxu} (see also (3) in Lemma \ref{lemmafund}), since $M^n$ has constant scalar curvature, one deduces that $|P|^2 = (\lambda-\rho) Tr(P).$ Besides, $Tr(P)$ and $|P|^2$ are also constants. So, we have
\begin{equation}
\label{lmn67012gh}
    u\Delta (Tr(Ric^3))=u\Delta(Tr(P^3)).
\end{equation}

We now need to obtain an expression for $\nabla_i R_{sj}R_{jl}R_{il}\nabla_s u$ in terms of $P.$ Indeed, one observes that
\begin{eqnarray}
\label{lmn67012gh22}
    \nabla_i R_{sj}R_{jl}R_{il}\nabla_s u&=&[\nabla_i(P_{sj}+\rho g_{sj})](P_{jl}+\rho g_{jl})(P_{il}+\rho g_{il})\nabla_s u\nonumber\\
    &=&\nabla_i P_{sj}P_{jl}P_{il}\nabla_s u+\nabla_i P_{sj}P_{jl}\rho g_{il}\nabla_s u+\nabla_i P_{sj}\rho g_{jl}P_{il}\nabla_s u\nonumber\\
    & &+\nabla_i P_{sj}\rho g_{jl}\rho g_{il}\nabla_s u\nonumber\\
    &=&\nabla_i P_{sj}P_{jl}P_{il}\nabla_s u+\rho\nabla_i P_{sj}P_{ji}\nabla_s u+\rho\nabla_i P_{sj}P_{ij}\nabla_s u\nonumber\\
    & &+\rho^2\nabla_i P_{si}\nabla_s u\nonumber\\
    &=&\nabla_i P_{sj}P_{jl}P_{il}\nabla_s u+2\rho\nabla_i P_{sj}P_{ij}\nabla_s u,
\end{eqnarray} where we have used that $\nabla_i P_{si}=0,$ which follows from the fact that $M$ has constant scalar curvature jointly with the twice contracted second Bianchi identity. Next, we compute 
\begin{eqnarray}
\label{lmn67012gh11}
    Ric^{2}_{ij}Ric^{2}_{ij}&=& R_{ik}R_{kj}R_{jl}R_{li}\nonumber\\
    &=& \left(P_{ik}P_{kj}+2\rho P_{ij}+\rho^{2} g_{ij}\right)\left(P_{il}P_{lj}+2\rho P_{ij}+\rho^{2} g_{ij}\right)\nonumber\\
    &=& P_{ik}P_{kj}P_{il}P_{lj} +4\rho P_{ik}P_{kj}P_{ji}+6\rho^2 P_{ij}P_{ij} + 4\rho^3 Tr(P)+\rho^4 n\nonumber\\
    &=&Tr(P^4)+4\rho Tr(P^3)+6\rho^2|P|^2+4\rho^3Tr(P)+n\rho^4
\end{eqnarray} and
\begin{eqnarray}
\label{lmn67012gh00}
    R_{ds}R_{dijs}R_{jl}R_{il} &=& (P_{ds}+\rho g_{ds})R_{dijs}(P_{jl}+\rho g_{jl})(P_{il}+\rho g_{il})\nonumber\\
    &=& (P_{ds}P_{jl}P_{il}+2\rho P_{ds}P_{ij}+\rho^2 P_{ds}g_{ij}+\rho g_{ds}P_{jl}P_{il}\nonumber\\&&+2\rho^2 g_{ds}P_{ij} +\rho^3 g_{ds}g_{ij})R_{dijs}\nonumber\\
    &=& P_{ds}R_{dijs}P_{jl}P_{il}+2\rho P_{ds}P_{ji}R_{dijs}-\rho^2 P_{ds}(P_{ds}+\rho g_{ds})\nonumber\\&& -\rho (P_{ij}+\rho g_{ij})P_{jl}P_{il}-2\rho^2 P_{ij}(P_{ij}+\rho g_{ij})-\rho^3 R\nonumber\\
    &=&P_{ds}R_{dijs}P_{jl}P_{il}+2\rho P_{ds}R_{dijs}P_{ij}-4\rho^2|P|^2-3\rho^3 Tr(P)\nonumber\\
    & &-\rho Tr(P^3)-\rho^3 R.
\end{eqnarray} 

At the same time, observe that 

\begin{eqnarray}
\label{lmn67012ghaa}
    \nabla_s R_{ij}\nabla_s R_{jl}R_{il}&=&\nabla_s(P_{ij}+\rho g_{ij})\nabla_s(P_{jl}+\rho g_{jl})(P_{il}+\rho g_{il})\nonumber\\
    &=&\nabla_s P_{ij}\nabla_s P_{jl} P_{il}+\rho \nabla_s P_{ij}\nabla_s P_{ij}.
\end{eqnarray} Moreover, as already mentioned, the constant scalar curvature condition implies that $|P|$ and $Tr(P)$ are also constants. Therefore, one deduces that 

\begin{eqnarray}
\label{lmn67012ghbb}
    \nabla_s(Tr(Ric^3))\nabla_s u=\nabla_s(Tr(P^3))\nabla_s u.
\end{eqnarray} Thereby, using (\ref{lmn67012gh}), jointly with (\ref{trR-trPaa}), (\ref{lmn67012gh22}), (\ref{lmn67012gh11}), (\ref{lmn67012gh00}), (\ref{lmn67012ghaa}) and  (\ref{lmn67012ghbb}) into Corollary \ref{corKA2}, one obtains that
\begin{eqnarray*}
    u\Delta Tr(P^3)&=&u\Delta Tr(Ric^3)\\
    &=&3(m+1)\left(\nabla_i P_{sj}P_{jl}P_{il}\nabla_s u+2\rho\nabla_i P_{sj}P_{ij}\nabla_s u\right)\\
    & &+\frac{3(m+1)u}{m}\left(Tr(P^4)+4\rho Tr(P^3)+6\rho^2|P|^2+4\rho^3Tr(P)+n\rho^4\right)\\
    & &+6u\left(P_{ds}R_{dijs}P_{jl}P_{il}+2\rho P_{ds}R_{dijs}P_{ij}-4\rho^2|P|^2-3\rho^3 Tr(P)-\rho Tr(P^3)-\rho^3 R\right)\\
    & &-(m+2)\nabla_s(Tr(P^3))\nabla_s u\\
    & &-\frac{3u}{m}\left(R-(m+n-2)\lambda\right)\left(Tr(P^3)+3\rho|P|^2+3\rho^2Tr(P)+n\rho^3\right)\\
    & &+\frac{3\lambda u}{m}\left(R-(n-1)\lambda\right)\left(|P|^2+2\rho Tr(P)+n\rho^2\right)\\
    & &+6u\left(\nabla_s P_{ij}\nabla_sP_{jl}P_{il}+\rho\nabla_s P_{ij}\nabla_s P_{ij}\right),
\end{eqnarray*} where we also used that $|Ric|^2=|P+\rho g|^2=|P|^2+2\rho Tr(P)+n\rho^2.$ Consequently, taking into account that $\rho^3R =-(m-1)\rho^4+\rho^3(n-1)\lambda$ and $R-(m+n-2)\lambda=-(m-1)(\rho+\lambda),$ we get

\begin{eqnarray*}
    u\Delta Tr(P^3)&=&3(m+1)\left(\nabla_i P_{sj}P_{jl}P_{il}\nabla_s u+2\rho\nabla_i P_{sj}P_{ij}\nabla_s u\right)\\
    & &+6u\left(\nabla_s P_{ij}\nabla_s P_{jl}P_{il}+\rho\nabla_s P_{ij}\nabla_s P_{ij}\right)\\
    & &+6u\left(P_{ds}R_{dijs}P_{jl}P_{il}+2\rho P_{ds}R_{dijs}P_{ij}\right)-(m+2)\nabla_s(Tr(P^3))\nabla_s u\\
    & &+\frac{3(m+1)u}{m}Tr(P^4)+\left(\frac{12(m+1)\rho u}{m}-6\rho u+\frac{3u}{m}(m-1)(\rho+\lambda)\right)Tr(P^3)\\
    & &+\left(\frac{18(m+1)\rho^2 u}{m}-24\rho^2u+\frac{9\rho u}{m}(m-1)(\rho+\lambda)-\frac{3(m-1)\lambda\rho u}{m}\right)|P|^2\\
    & &+\left(\frac{12(m+1)\rho^3 u}{m}-18\rho^3 u+\frac{9\rho^2 u}{m}(m-1)(\rho+\lambda)-\frac{6(m-1)\lambda\rho^2 u}{m}\right)Tr(P)\\
    & &+\left(\frac{3(m+1)n\rho^4 u}{m}-6u\rho^3(-(m-1)\rho+(n-1)\lambda)+\frac{3n\rho^3 u}{m}(m-1)(\rho+\lambda)\right.\\
    & &\left.-\frac{3(m-1)n\lambda\rho^3 u}{m}\right).
\end{eqnarray*}
Simplifying the last four terms in the right hand side of the above expression, we achieve
\begin{eqnarray*}
    u\Delta Tr(P^3) &=& 3(m+1)\left(\nabla_i P_{sj}P_{jl}P_{il}\nabla_s u+2\rho\nabla_i P_{sj}P_{ij}\nabla_s u\right)\\
    & &+6u\left(\nabla_s P_{ij}\nabla_s P_{jl}P_{il}+\rho\nabla_s P_{ij}\nabla_s P_{ij}\right)\\
    & &+6u\left(P_{ds}R_{dijs}P_{jl}P_{il}+2\rho P_{ds}R_{dijs}P_{ij}\right)- (m+2)\nabla_s(Tr(P^3))\nabla_s u\\
    & &+\frac{3(m+1)u}{m}Tr(P^4)+\frac{3u}{m}\left(3(m+1)\rho+(m-1)\lambda\right)Tr(P^3)\\
    & &+\frac{3\rho u}{m}\left((m+3)\rho+2(m-1)\lambda\right)|P|^2\\
    & &+\frac{3\rho^2 u}{m}\left((m+1)\rho+(m-1)\lambda\right)Tr(P)\\
    & &+ 6\rho^3 u \left((m+n-1)\rho-(n-1)\lambda\right),
\end{eqnarray*} which finishes the proof of the lemma. 

\end{proof}

\medskip

\section{The Proof of Theorem \ref{theo1}, Theorem \ref{theo3} and Corollary \ref{theo2}}
\label{secProofs}

In this section, we are going to present the proof of Theorem \ref{theo1}, Theorem \ref{theo3} and Corollary \ref{theo2}.

\subsection{Proof of Theorem \ref{theo1}}

\begin{proof}
In the first part of the proof, we shall follow Proposition 3.13 of \cite{Petersen-Chenxu}. To begin with, denoting $\alpha=\frac{ R+(m-n)\lambda}{m(m-1)}$ and $\widetilde{\mu}=\frac{\mu}{m-1},$ one sees from \eqref{transnormal} that
\begin{eqnarray*}
	\frac{|\nabla u|^2}{\widetilde{\mu}-\alpha u^2}=1,
\end{eqnarray*} which defines the distance function $r=\frac{1}{\sqrt{\alpha}}\arccos\left(\frac{u}{\sqrt{\widetilde{\mu}\alpha^{-1}}}\right).$ In particular, the potential function can be recovered as $u(r)=\sqrt{\widetilde{\mu}\alpha^{-1}}\cos(\sqrt{\alpha}r).$ From Remark \ref{remK}, the set of critical points for $u$ coincides with the set of maximum values, namely, $Crit(u)=MAX(u).$ Thus, we may identify $MAX(u)=r^{-1}(0).$ So, following the argument in \cite[Lemma 7]{Wang} with the appropriate adaptation, and using that $u$ vanishes on each boundary component, we deduce that each connected component of $MAX(u)$ is a smooth submanifold. It then follows from Lemma \ref{lemmaD} that
\begin{eqnarray}\label{laplacian r}
	\Delta r=\textrm{tr}(\mathcal{A}_\theta)+\frac{n-k-1}r+O(r),
\end{eqnarray} where $k$ is the dimension of a connected component $N$ of $MAX(u)$ and $A_{\theta}$ stands for the second fundamental form with respect to $\theta.$ By \eqref{mqE}, without loss of generality, we may multiply the potential function $u$ by a constant $\beta$ so that  $\beta\, u$ is a potential function for the same metric and constant $\lambda$ as $u.$ In view of this, we can assume that $u(r)=\cos(\sqrt{\alpha}r)$ and consequently, we deduce
\begin{eqnarray*}
	\nabla_i\nabla_j u=-\sqrt{\alpha}\sin(\sqrt{\alpha} r)\nabla_i\nabla_j r-\alpha\cos(\sqrt{\alpha}r)\nabla_i r\nabla_j r 
\end{eqnarray*}
and 
\begin{eqnarray}\label{delta-u}
	\Delta u=-\sqrt{\alpha}\sin(\sqrt{\alpha}r)\Delta r-\alpha\cos(\sqrt{\alpha}r)|\nabla r|^2.
\end{eqnarray} Taking into account the Taylor expansions, around $r=0,$

\begin{eqnarray*}
\sin(\sqrt{\alpha}r)=\sqrt{\alpha} r+O(r^3)\;\,\,\,\,\mathrm{and}\;\,\,\,\,\,\cos(\sqrt{\alpha}r)=1+O(r^2),
\end{eqnarray*} we obtain from  \eqref{laplacian r} and \eqref{delta-u} that

\begin{eqnarray}\nonumber
	\Delta u&=&(-\alpha r+O(r^3))\left(\textrm{tr}(\mathcal{A}_\theta)+\frac{n-k-1}r+O(r)\right)+(-\alpha+O(r^2))\\\label{du-o}
	&=&-\alpha(n-k)+O(r).
\end{eqnarray}

It is known from \eqref{tensorP} that $P=Ric-\frac{(n-1)\lambda-R}{m-1}g.$ In particular, by setting $\rho=\frac{(n-1)\lambda-R}{m-1},$ we may write \eqref{eq-laplacian} in terms of $P$ and $\rho,$ at the connected component $N$ of $MAX(u),$ as

\begin{eqnarray}\label{du-trp}
	\Delta u=\frac{1}{m}(Tr(P)-n(\lambda-\rho)),
\end{eqnarray} where we have used that $u\mid_{N}=1.$ Then, since $\alpha=\frac{\lambda-\rho}{m},$ we combine \eqref{du-o}, restricted to $N,$ and \eqref{du-trp} in order to infer
\begin{eqnarray*}
	Tr(P)=k(\lambda-\rho).
\end{eqnarray*} 

We now claim that tangent and normal vector fields to $N$ are the eigenvectors corresponding to $\lambda-\rho$ and $0,$ respectively. Indeed, given a point $p\in N$ and $X\in \mathfrak{X}(N)$ a tangent vector at $p,$ since $\nabla u\mid_N =0,$ we have

\begin{equation*}
\nabla^2 u(X)(p)=\nabla_{X}\nabla u (p)=0,
\end{equation*} where we have used the fact that $\nabla_{X} \nabla u(p)$ only depends on the value of $X(p)$ and $\nabla u$ along of a curve through $p$ with $X$ as a tangent vector at $p.$ Hence, by using  (\ref{mqE}), we obtain

\begin{equation*}
0=\nabla_{X}\nabla u (p)=\frac{u}{m}\left(P(X)-(\lambda-\rho)X\right).
\end{equation*} Consequently, $P(X)=(\lambda -\rho)X,$ for all $X\in \mathfrak{X}(N)$ and therefore, the tangent vectors to $N$ corresponds to the eigenvalue $\lambda-\rho$ for $P.$ Besides, it follows from assertion (2) of Proposition \ref{prop-p} that, at $Crit(u)$, 
\begin{eqnarray*}
	P\circ(P-(\lambda-\rho)I)=0.
\end{eqnarray*} Thus, the only possible eigenvalues for $P$ at $N$ are $\lambda-\rho$ and $0.$ Moreover, since $Tr(P)=k(\lambda-\rho)$ and $k=dim(N),$ one concludes that normal vectors to $N$ correspond to the eigenvalue $0.$

Proceeding, one concludes that
\begin{eqnarray*}
    P\mid_N=\left(\begin{array}{cc}
        (\lambda-\rho)I_k & 0 \\
        0 & [\textbf{0}]_{n-k}
    \end{array}\right)
\end{eqnarray*}
is the $n\times n$ matrix of the tensor $P$ at the manifold $N$. In terms of the Ricci tensor, we have 
\begin{eqnarray}\label{eq-trace}
    Ric\mid_N=\left(\begin{array}{cc}
        \lambda I_k & 0 \\
        0 & \frac{(n-1)\lambda- R}{m-1} I_{n-k}
    \end{array}\right).
\end{eqnarray} In particular, taking the trace in \eqref{eq-trace}, we see that
\begin{eqnarray*}
     R=\frac{k(m-n)+n(n-1)}{m+n-k-1}\lambda,
\end{eqnarray*}
for some $k\in \{0,1\ldots,n-1\},$ where we have also used that $R<n\lambda$ (see Remark \ref{remJK}). So, the proof is finished.
\end{proof}

\subsection{Proof of Theorem \ref{theo3}}

\begin{proof}
Since $R=(n-1)\lambda,$ it follows from \eqref{eq-ridu} that the eigenvalue $\lambda_1$ associated to the eigenvector $\nabla u$ for the Ricci tensor is zero. We now need to show that all non-zero eigenvalues of the Ricci tensor are equal to $\lambda.$ Before doing so, we first claim that

\begin{eqnarray}\label{eq-(n-1)}
    |\mathring{Ric}|^2=\frac{R^2}{n(n-1)}.
\end{eqnarray} Indeed, since $R$ is constant, assertion (3) in Lemma \ref{lemmafund} (see also \cite[Lemma 3.2]{CaseShuWei}) yields 

    $$(m-1)|\mathring{Ric}|^2=-\frac{m+n-1}{n}( R-n\lambda)\left( R-\frac{n(n-1)}{m+n-1}\lambda\right).$$ Substituting $R=(n-1)\lambda$ into the above expression, we obtain
\begin{eqnarray}
(m-1)|\mathring{Ric}|^2 &=&-R^2 \frac{(m+n-1)}{n}\left(1-\frac{n}{n-1}\right)\left(1-\frac{n}{m+n-1}\right),
\end{eqnarray} which immediately implies $$|\mathring{Ric}|^2=\frac{R^2}{n(n-1)},$$ as claimed.

Let $\lambda_i,$ $i\neq 1,$ denote the possible non-zero eigenvalues of the Ricci tensor. Then
\begin{equation*}
\sum_{i=2}^{n}(\lambda_i-\lambda)^2 = |Ric|^2-2\lambda R+(n-1)\lambda^2
=|\mathring{Ric}|^2-\frac{R^2}{n(n-1)},
\end{equation*} where we have used the identities $|\mathring{Ric}|^2=|Ric|^2-\frac{R^2}{n}$ and $R=(n-1)\lambda.$ Therefore, one obtains from \eqref{eq-(n-1)} that $\lambda_i=\lambda,$ for $i=2, \ldots ,n,$ that is, the eigenvalues of the Ricci are all constants with $\lambda_2=\ldots=\lambda_n=\lambda.$ Consequently, Corollary \ref{coro-tensort} ensures that $T\equiv 0.$ In particular, since the Ricci tensor is parallel, the Cotton tensor (\ref{cot1}) also vanishes. Thus, by Lemma \ref{l1}, we have $W_{ijkl}\nabla_l u=0.$ We are therefore in a position to invoke Theorem 1.2 of \cite{He-Petersen-Wylie2012} to infer that the metric splits off as $g=dt^2+\varphi^2(t)\widetilde{g}_{_N},$ where $\widetilde{g}_N$ is $\kappa$-Einstein with non-negative Ricci curvature and $u=u(t).$

In view of \eqref{eq-ridu}, we get
   $$Ric(\nabla u,\nabla u)=\frac{(n-1)\lambda-R}{m-1}(u')^2=0$$ and hence, we may apply Proposition \ref{laux} to infer

\begin{eqnarray*}
    0=Ric(\nabla u,\nabla u)=(u')^2 Ric(\partial t,\partial t)=-(u')^2\frac{(n-1)}{\varphi}\varphi''.
\end{eqnarray*} Since $u$ is analytical in harmonic coordinates (and $u$ is not constant), we conclude that $\varphi''(t)/\varphi(t)=0,$ which implies that  $\varphi(t)=c$ or $\varphi(t)=c t$, for some positive constant $c.$ However, as mentioned in Section 2.1, the second case can not hold.

Proceeding, since $g=dt^2+c^2 \widetilde{g}_{_N}$ and $\widetilde{g}_{_N}$ is a $\kappa$-Einstein metric, we may use again Proposition \ref{laux} to deduce
\begin{eqnarray*}
    Ric(V,W)=\kappa\, \widetilde{g}_{_N}(V,W).
\end{eqnarray*} Consequently, the scalar curvature is $R=\frac{\kappa}{c^2}(n-1)$ and moreover, $\lambda=\frac{\kappa}{c^2}$ and $(N^{n-1},\,g_{_N})$ is $\lambda$-Einstein manifold, where $g_{_N}=c^2\,\widetilde{g}_{_N}.$

Finally, observe that, by \eqref{Hlie} and the fact that $Ric=\lambda g_{_N},$ the potential function $u=u(t)$ satisfies
\begin{eqnarray*}
    u''(t)dt^2=\nabla^2 u=\frac{u}{m}(Ric-\lambda g)=-\lambda\frac{u}{m}dt^2,
\end{eqnarray*}
with the boundary condition $u\mid_{\partial M}=0.$ Hence, without loss of generality, we may take $u(t)=\sin\left(\frac{\sqrt{\lambda}}{\sqrt{m}}t\right).$
From this, it follows that $M^n$ is isometric, up to scaling, to the cylinder $\left[0,\frac{\sqrt{m}}{\sqrt{\lambda}}\pi\right]\times N,$ where $N$ is a compact $\lambda$-Einstein manifold. This finishes the proof of the theorem. 
\end{proof}

\vspace{0.30cm}

Next, we establish a key proposition, valid for arbitrary dimensions $n\geq 3,$ which will be used in the proofs of Corollary \ref{theo2} and Corollary \ref{corA}.

\begin{proposition}
 \label{propK11} 
There is no compact nontrivial quasi-Einstein manifold $M^n$ with boundary and constant scalar curvature $R=\frac{m+n(n-2)}{m+n-2}\lambda.$
\end{proposition}
\begin{proof}
 We argue by contradiction, assuming that a compact nontrivial quasi-Einstein manifold $M^n$ with boundary has constant scalar curvature $R=\frac{m+n(n-2)}{m+n-2}\lambda,$ which corresponds the case $k=1$ in Theorem \ref{theo1}. Hence, by the work of Wang \cite{Wang} (see also \cite[Theorem 1.1]{GT} and \cite[Theorem 6.1]{MP}), one obtains that  $MAX(u)$ is a focal variety of the isoparametric function $u$ of dimension one and connected (see \cite[Theorem 2.2]{GT2}). So $MAX(u)$ is  totally geodesic. This therefore implies that $MAX(u)=\Bbb{S}^1$ and consequently, $M$ is homotopic to $\Bbb{S}^1$ (see \cite{Miyaoka}), which leads to a contradiction with the fact that $M^n$ has finite fundamental group (see Remark \ref{remarkpifinite}). Thus, the proof is completed. 
\end{proof}

\subsection{Proof of Corollary \ref{theo2}}

\begin{proof}
To begin with, we invoke Theorem \ref{theo1} and Proposition \ref{propK11} to infer that the scalar curvature is either $R=\frac{6}{m+2}\lambda$ or $R=2\lambda.$ In the first case, it suffices to use Proposition 2.4 in  \cite{Petersen-Chenxu}  to conclude that $(M^3,\,g)$ is isometric to the standard hemisphere $\mathbb{S}^{3}_{+}.$ In the second case, when $R=2\lambda,$ we can apply Theorem \ref{theo3} to infer that $(M^3,\,g)$ is isometric, up to scaling, to the cylinder $I\times N,$ where $N$ is a compact $\lambda$-Einstein manifold. Moreover, from \eqref{eqe1} and the Killing–Hopf theorem, we deduce that $N=\mathbb{S}^2.$ This completes the proof of Corollary \ref{theo2}.

\end{proof}

\section{The Proof of Theorem \ref{theodim4} and Corollary \ref{corA}}
\label{secthm4}

In this section, we present the proofs of Theorem \ref{theodim4} and Corollary \ref{corA}. In the first part, we follow the approach developed by Cheng and Zhou in \cite{ChengZhou}. To this end, we first establish the following proposition.

\begin{proposition}
\label{propKK1}
    Let $(M^4,\,g,\,u,\,\lambda)$ be an $m$-quasi-Einstein manifold with $m>1$ and constant scalar curvature $R=\frac{2(m+2)\lambda}{m+1}.$ Then we have
\begin{eqnarray}
\label{eqpk1p}
u\Delta Tr(P^3)+(m+2)\langle \nabla(Tr(P^3)),\,\nabla u\rangle&=& 6u\lambda Tr(P^3) + 6\frac{\lambda^2}{m+1}u |P|^2 \nonumber\\&&+6u\left(\nabla_s P_{ij}\nabla_s P_{jl}P_{il}+\rho\nabla_s P_{ij}\nabla_s P_{ij}\right)\nonumber\\
    &&+6u\left(P_{ds}R_{dijs}P_{jl}P_{il}+2\rho P_{ds}R_{dijs}P_{ij}\right)\\& &+12\rho^4 m^2 (m+1)u.\nonumber
\end{eqnarray} 
\end{proposition}

\begin{proof}
Initially, let $\mu_i$ be the eigenvalues of $P$ defined in \eqref{tensorP} with respect to the adapted orthonormal frame $\{e_i\}_{i=1}^{4}$ so that $e_1=-\frac{\nabla u}{|\nabla u|}.$ In particular, it follows from \eqref{eq-ridu} that $\mu_1=0.$ Consequently,
\begin{eqnarray*}
Tr(P)=\mu_2+\mu_3+\mu_4\,\,\,\,\,\,\,\,\,\hbox{and}\,\,\,\,\,\,\,\,\,\,\,|P|^2=\mu_{2}^{2}+\mu_{3}^{2}+\mu_{4}^{2},
\end{eqnarray*} where $P=Ric-\frac{3\lambda-R}{m-1}g.$ Thus, for $R=\frac{2(m+2)}{m+1}\lambda$, it follows from \eqref{tensorP} that 
  \begin{eqnarray}
  \label{trp}
     Tr(P)&=&\frac{(m+n-1)R-n(n-1)\lambda}{m-1}= \frac{(m+3)R-12\lambda}{m-1}=\frac{2m}{m+1}\lambda,
  \end{eqnarray} which implies that $Tr(P)$ is a positive constant.

Next, by Proposition 3.3 in \cite{Petersen-Chenxu}, one has $|P|^2=\left(\lambda-\rho\right)Tr(P),$ where $\rho=\frac{3\lambda-R}{m-1}.$ This combined with  \eqref{trp} yields

\begin{eqnarray}
\label{np}
    |P|^2&=&\frac{(m-4)\lambda+R}{m-1}\,Tr(P)=\frac{m}{m+1}\lambda\, Tr(P)=\frac{1}{2}(Tr(P))^2.
\end{eqnarray}

On the other hand, by simplifying the last three terms in the right hand side of Lemma \ref{lemKA3}, taking into account that $\rho=\frac{\lambda}{m+1},$ $Tr\,(P)=\frac{2m}{m+1}\lambda,$ $2|P|^2=(Tr\,(P))^2$ and $n=4,$ one deduces that

\begin{eqnarray}
\label{plm1az}
    u\Delta Tr(P^3) &=& 3(m+1)\left(\nabla_i P_{sj}P_{jl}P_{il}\nabla_s u+2\rho\nabla_i P_{sj}P_{ij}\nabla_s u\right)\nonumber\\
    & &+6u\left(\nabla_s P_{ij}\nabla_s P_{jl}P_{il}+\rho\nabla_s P_{ij}\nabla_s P_{ij}\right)\nonumber\\
    & &+6u\left(P_{ds}R_{dijs}P_{jl}P_{il}+2\rho P_{ds}R_{dijs}P_{ij}\right)- (m+2)\nabla_s(Tr(P^3))\nabla_s u\nonumber\\
    & &+\frac{3(m+1)u}{m}Tr(P^4)+\frac{3u}{m}\left(3(m+1)\rho+(m-1)\lambda\right)Tr(P^3)\nonumber\\
    & &+12\rho^4 m^2 (m+1)u.
\end{eqnarray} At the same time, since $P_{sj}\nabla_s u=P(\nabla u)=0,$ we have from (\ref{mqE}) that 

\begin{eqnarray*}
0&=&\nabla_{i}(P_{sj}\nabla_s u)= \nabla_{i}P_{sj}\nabla_{s}u+\frac{u}{m}P_{sj}(R_{is}-\lambda g_{is})\nonumber\\&=&\nabla_{i}P_{sj}\nabla_{s}u+\frac{u}{m}P_{sj}P_{is} -\frac{(\lambda-\rho)}{m}uP_{ij}
\end{eqnarray*} so that 

\begin{equation} 
\label{eq8901a}
\nabla_{i}P_{sj}\nabla_{s}u=-\frac{u}{m}P_{sj}P_{is} +\frac{(\lambda-\rho)}{m}uP_{ij}.
\end{equation} Hence, the first term in the right hand side of (\ref{plm1az}) becomes

\begin{eqnarray*}
\Lambda&=&3(m+1)\left(\nabla_i P_{sj}P_{jl}P_{il}\nabla_s u+2\rho\nabla_i P_{sj}P_{ij}\nabla_s u\right)\nonumber\\
&=& 3(m+1)\left(-\frac{u}{m}P_{sj}P_{is}  + \frac{(\lambda-\rho)}{m}u P_{ij}\right)\left(P_{jl}P_{il}+2\rho P_{ij}\right)\nonumber\\&=& 3(m+1)\left( -\frac{u}{m}(Tr(P^4))+ \frac{(\lambda-3\rho)}{m}u\, Tr(P^3)+2\frac{\rho(\lambda-\rho)}{m}u|P|^2\right).
\end{eqnarray*} Substituting this into (\ref{plm1az}) and rearranging terms, one sees that

\begin{eqnarray}
u\Delta Tr(P^3)+(m+2)\langle \nabla(Tr(P^3)),\,\nabla u\rangle&=& 6u\lambda Tr(P^3) + 6\frac{\lambda^2}{m+1}u |P|^2 \nonumber\\&&+6u\left(\nabla_s P_{ij}\nabla_s P_{jl}P_{il}+\rho\nabla_s P_{ij}\nabla_s P_{ij}\right)\nonumber\\
    &&+6u\left(P_{ds}R_{dijs}P_{jl}P_{il}+2\rho P_{ds}R_{dijs}P_{ij}\right)\nonumber\\& &+12\rho^4 m^2 (m+1)u.
\end{eqnarray} This concludes the proof of the proposition. 
\end{proof}

\vspace{0.30cm}

In order to proceed, we need to prove the following result.

\begin{proposition}
\label{propKK2}
    Let $(M^4,\,g,\,u,\,\lambda)$ be an $m$-quasi-Einstein manifold with $m>1$ and constant scalar curvature $R=\frac{2(m+2)\lambda}{m+1}.$ Then we have:
    \begin{eqnarray}
    \label{idP1}
        u L_{m+2}(Tr(P^3))&=&8(m+1)\rho u Tr(P^3)+6u\nabla_s P_{ij}\nabla_s P_{jl}P_{il}- 3m\rho u|\nabla P|^2\nonumber\\
        & &-16m^3(m+1)\rho^4 u
    \end{eqnarray}
    and
    \begin{eqnarray}
    \label{idP2}
        u L_{m+2}(Tr(P^3))&\geq&  8(m+1)\rho u Tr(P^3)- 3m\rho u|\nabla P|^2\nonumber\\
        & &-16m^3(m+1)\rho^4 u,
    \end{eqnarray}
    where $u L_{m+2}(f)=u\Delta f+(m+2)\langle\nabla f,\nabla u\rangle$ and $\rho=\frac{\lambda}{m+1}$.
\end{proposition}

\begin{proof} First of all, observe that our assumption is equivalent to $R=2(m+2)\rho,$ where $\rho=\frac{\lambda}{m+1}.$ Moreover, one sees that
    \begin{eqnarray}
    \label{eqkmb396pi}
    Tr(P)=2m\rho\;\;\;\; \mathrm{and}\;\;\;\; |P|^2=2m^2\rho^2=\frac{1}{2}(Tr(P))^2.    
    \end{eqnarray} Now, we need to compute $u L_{m+2}(|P|^2).$ To this end, since $Ric=P+\rho g,$ we notice from Lemma \ref{lemKA1} that
    \begin{eqnarray*}
        u(\Delta P_{ik})&=&\nabla_i P_{sk}\nabla_s u+m\nabla_k P_{is}\nabla_s u+\frac{m+1}{m}u(P_{is}+\rho g_{is})(P_{sk}+\rho g_{sk})\\
        & &+2uR_{jiks}(P_{js}+\rho g_{js})-(m+2)\nabla_s P_{ik}\nabla_s u\\
        & &+\frac{u}{m}(m-1)(m+2)\rho(P_{ik}+\rho g_{ik})-\frac{u}{m}(m-1)(m+1)\rho^2 g_{ik},
    \end{eqnarray*} where we have used that $n=4,$ $R-(m+n-2)\lambda=-(m-1)(m+2)\rho$ and $\lambda(R-(n-1)\lambda)=-(m-1)(m+1)\rho^2.$ Next, expanding the expression in the right hand side and rearranging terms, we have
    \begin{eqnarray}\nonumber
        u L_{m+2}(P_{ik})&=&\nabla_i P_{sk}\nabla_s u+m\nabla_k P_{is}\nabla_s u+\frac{m+1}{m}u P_{ik}^{2}+\frac{2(m+1)\rho u}{m}P_{ik}\\\nonumber
        & &+\frac{(m+1)\rho^2 u}{m}g_{ik}+2u R_{jiks}P_{js}-2\rho u P_{ik}-2\rho^2 u g_{ik}\\\nonumber
        & &+\frac{(m-1)(m+2)\rho u}{m}P_{ik}+\frac{(m-1)\rho^2 u}{m}g_{ik}\\\label{eqi1}
        &=&\nabla_i P_{sk}\nabla_s u+m\nabla_k P_{is}\nabla_s u+\frac{m+1}{m}u P_{ik}^{2}+(m+1)\rho u P_{ik}+2u R_{jiks}P_{js}.
    \end{eqnarray}

Proceeding, we use that $\lambda=(m+1)\rho$ and Eq. (\ref{eq8901a}) to infer 
    \begin{eqnarray*}
      \nabla_i P_{sk}\nabla_s u=-\frac{u}{m}(P^{2}_{ik}-m\rho P_{ik}).
    \end{eqnarray*}
    Consequently,
    \begin{eqnarray*}
        \nabla_i P_{sk}\nabla_s u+m\nabla_k P_{is}\nabla_s u=-\frac{(m+1)u}{m}(P_{ik}^{2}-m\rho P_{ik}).
    \end{eqnarray*}
    This allow us to rewrite \eqref{eqi1} as
    \begin{eqnarray*}
        u L_{m+2}(P_{ik})&=&-\frac{(m+1)u}{m}P_{ik}^{2}+(m+1)\rho u P_{ik}+\frac{(m+1)u}{m}P_{ik}^{2}+(m+1)\rho u P_{ik}+2 u R_{jiks}P_{js}\\
        &=&2(m+1)\rho u P_{ik}+2u R_{jiks}P_{js}.
    \end{eqnarray*}

 At the same time, by using that $u L_{m+2}(P_{ik})=u\Delta P_{ik}+(m+2)\langle\nabla P_{ik},\nabla u\rangle,$ we infer
    \begin{eqnarray*}
        u L_{m+2}(|P|^2)&=&u L_{m+2}(P_{ik}P_{ik})\\
        &=&u\Delta (P_{ik}P_{ik})+(m+2)\langle\nabla(P_{ik}P_{ik}),\nabla u\rangle\\
        &=&u(2P_{ik}\Delta P_{ik}+2|\nabla P|^2)+2(m+2)P_{ik}\langle\nabla P_{ik},\nabla u\rangle\\
        &=&2u|\nabla P|^2+2u P_{ik}L_{m+2}(P_{ik})\\
        &=&2u|\nabla P|^2+4(m+1)\rho u |P|^2+4uP_{ik}R_{jiks}P_{js}.
    \end{eqnarray*} Besides, since $|P|^2$ is constant, then $uL_{m+2}(|P|^2)=0$ and hence, we have
    \begin{eqnarray}\label{eqi2}
        uP_{ik}R_{jiks}P_{js}=-\frac{u}{2}|\nabla P|^2-(m+1)\rho u|P|^2.
    \end{eqnarray}

 On the other hand, it follows from (\ref{eqpk1p}) that
    \begin{eqnarray}\nonumber
        u L_{m+2}(Tr(P^3))&=&6(m+1)\rho u Tr(P^3)+6(m+1)\rho^2 u|P|^2\\\nonumber
        & &+6u(\nabla_s P_{ij}\nabla_s P_{jl} P_{il}+\rho\nabla_s P_{ij}\nabla_s P_{ij})\\\nonumber
        & &+6u(P_{ds}R_{dijs}P_{jl}P_{il}+2\rho P_{ds}R_{dijs}P_{ij})\\\label{eqii}
        & &+12m^2 (m+1)\rho^4 u.
    \end{eqnarray} To proceed, we need to deal with the terms that depend of the Riemannian curvature. Thereby, fix a point $p\in M$ and assume $P_{ij}=\mu_i \delta_{ij}$ at $p,$ that is, $\mu_i,$ $i=1,2,3,4$ are the eigenvalues of the tensor $P$ at $p$ and recall that $\mu_{1}=0.$ Hence, one easily verifies that
    \begin{eqnarray*}
        P_{ds}R_{dijs}P_{jl}P_{il}=\sum_{j=2}^{4}\sum_{d=2}^{4}\mu_d R_{djjd}\mu_{j}^{2}.
    \end{eqnarray*}
    Denoting $K_{dj}=R_{djdj}$, it follows that
    \begin{eqnarray}\nonumber
        P_{ds}R_{dijs}P_{jl}P_{il}&=&-\mu_2 K_{23}\mu_{3}^{2}-\mu_2 K_{24}\mu_{4}^{2}-\mu_3 K_{32}\mu_{2}^{2}-\mu_3 K_{34}\mu_{4}^{2}-\mu_4 K_{42}\mu_{2}^{2}-\mu_4 K_{43}\mu_{3}^{2}\\\nonumber
        &=&-K_{23}\mu_2\mu_3(\mu_3+\mu_2)-K_{24}\mu_2\mu_4(\mu_2+\mu_4)-K_{34}\mu_3\mu_4(\mu_3+\mu_4)\\\nonumber
        &=&-K_{23}\mu_2\mu_3(Tr(P)-\mu_4)-K_{24}\mu_2\mu_4(Tr(P)-\mu_3)-K_{43}\mu_4\mu_3(Tr(P)-\mu_2)\\\label{eqi3}
        &=&-Tr(P)(K_{23}\mu_2\mu_3+K_{34}\mu_3\mu_4+K_{24}\mu_2\mu_4)+(K_{23}+K_{34}+K_{24})\mu_2\mu_3\mu_4.
    \end{eqnarray} Moreover, notice that 
    \begin{eqnarray*}
        R_{22}+R_{33}+R_{44}=R-R_{11}=R-\rho=Tr(P)+3\rho
    \end{eqnarray*}
    and
    \begin{eqnarray*}
        K_{12}+K_{13}+K_{14}=R_{11}=\rho,
    \end{eqnarray*}
    which therefore implies that 
    \begin{eqnarray*}
        Tr(P)+3\rho=R-R_{11}=R_{22}+R_{33}+R_{44}=2(K_{23}+K_{34}+K_{24})+R_{11}.
    \end{eqnarray*} Besides, $K_{23}+K_{34}+K_{24}=\frac{1}{2}(Tr(P)+2\rho)=(m+1)\rho.$ In view of this, we may rewrite \eqref{eqi3} as 
    \begin{eqnarray*}
        P_{ds}R_{dijs}P_{jl}P_{il}=-2m\rho(K_{23}\mu_2\mu_3+K_{34}\mu_3\mu_4+K_{24}\mu_2\mu_4)+(m+1)\rho\mu_2\mu_3\mu_4.
    \end{eqnarray*} Similarly, one easily verifies that
    \begin{eqnarray}
    \label{eqiii}
        P_{ds}R_{dijs}P_{ij}&=&\sum_{d=2}^{4}\sum_{j=2}^{4}\mu_d R_{djjd}\mu_j =-2(K_{23}\mu_2\mu_3+K_{24}\mu_2\mu_4+K_{34}\mu_3\mu_4).
    \end{eqnarray}
    Hence, Eq. \eqref{eqii} becomes
    \begin{eqnarray*}
        u L_{m+2}(Tr(P^3))&=&6(m+1)\rho u Tr(P^3)+6(m+1)\rho^2 u|P|^2+6u(\nabla_s P_{ij}\nabla_s P_{jl}P_{il}+\rho |\nabla P|^2)\\
        & &-12(m+2)\rho u(K_{23}\mu_2\mu_3+K_{24}\mu_2\mu_4+K_{34}\mu_3\mu_4)+6(m+1)\rho u\mu_2\mu_3\mu_4\\
        & &+12m^2 (m+1)\rho^4 u\\
        &=&6(m+1)\rho u Tr(P^3)+12m^2(m+1)\rho^4 u+6u(\nabla_s P_{ij}\nabla_s P_{jl}P_{il}+\rho |\nabla P|^2)\\
        & &-12(m+2)\rho u(K_{23}\mu_2\mu_3+K_{24}\mu_2\mu_4+K_{34}\mu_3\mu_4)+6(m+1)\rho u\mu_2\mu_3\mu_4\\
        & &+12m^2 (m+1)\rho^4 u\\
        &=&6(m+1)\rho u Tr(P^3)+6u(\nabla_s P_{ij}\nabla_s P_{jl}P_{il}+\rho |\nabla P|^2)\\
        & &-12(m+2)\rho u(K_{23}\mu_2\mu_3+K_{24}\mu_2\mu_4+K_{34}\mu_3\mu_4)\\
        & &+6(m+1)\rho u\mu_2\mu_3\mu_4+24 m^2 (m+1)\rho^4 u,
    \end{eqnarray*} where we used that $|P|^2 =2m^2 \rho^2.$ Moreover, by combining \eqref{eqi2} and \eqref{eqiii}, we arrive at
    \begin{eqnarray*}
        u(K_{23}\mu_2\mu_3+K_{24}\mu_2\mu_4+K_{34}\mu_3\mu_4)=\frac{u|\nabla P|^2}{4}+m^2(m+1)\rho^3u.
    \end{eqnarray*} Consequently,
    
    \begin{eqnarray}
    \label{klmcx1}
        u L_{m+2}(Tr(P^3))&=&6(m+1)\rho u Tr(P^3)+6u(\nabla_s P_{ij}\nabla_s P_{jl}P_{il}+\rho |\nabla P|^2)\nonumber\\
        & &-3(m+2)\rho u|\nabla P|^2-12m^2(m+2)(m+1)\rho^4 u\nonumber\\
        & &+6(m+1)\rho u\mu_2\mu_3\mu_4+24 m^2 (m+1)\rho^4 u\nonumber\\
        &=&6(m+1)\rho u Tr(P^3)+6u\nabla_s P_{ij}\nabla_s P_{jl}P_{il}-3m\rho u|\nabla P|^2\nonumber\\
        & &+6(m+1)\rho u\mu_2\mu_3\mu_4-12m^3 (m+1)\rho^4 u.
    \end{eqnarray}

    At the same time, similar to \cite[pg. 264]{ChengZhou}, by letting $\alpha=\mu_2,$ $\beta=\mu_3$ and $\kappa=\mu_4$ in the following algebraic identity
\begin{eqnarray*}
    (\alpha+\beta+\kappa)^3=3(\alpha+\beta+\kappa)(\alpha^2+\beta^2+\kappa^2)-2(\alpha^3+\beta^3+\kappa^3)+6\alpha \beta \kappa,
\end{eqnarray*} we obtain
    \begin{eqnarray*}
        (Tr(P))^3=3|P|^2 Tr(P)-2Tr(P^3)+6\mu_2\mu_3\mu_4.
    \end{eqnarray*} Of which, 
    \begin{equation}
    \label{eqTYU8}
    3\mu_2\mu_3\mu_4=Tr(P^3)-2m^3\rho^3.
    \end{equation} This substituted into (\ref{klmcx1}) yields
    \begin{eqnarray*}
        u L_{m+2}(Tr(P^3))&=&8(m+1)\rho u Tr(P^3)+6u\nabla_s P_{ij}\nabla_s P_{jl}P_{il}-3m\rho u|\nabla P|^2\\
        & &-16m^3 (m+1)\rho^4 u,
    \end{eqnarray*} which proves \eqref{idP1}.

    Finally, for the fixed orthonormal frame, by using (\ref{eqkmb396pi}) and Lemma \ref{lemsum}, one deduces that $\mu_i\geq 0,$ for all $i.$ Hence, $\nabla_s P_{ij}\nabla_s P_{jl}P_{il}=|\nabla P_{ii}|^2\mu_i\geq 0$ and this proves the second assertion \eqref{idP2}. 
\end{proof}

We now establish the following essential lemma, which yields a key inequality involving the nonnegative function $|\nabla u|^2(Tr(P^3)-2m^3\rho^3).$ As mentioned earlier, this result plays a crucial role in the proof of Theorem~\ref{theodim4}.

\begin{lemma}
\label{lemKAp}
    Let $(M^4,\,g,\,u,\,\lambda)$ be an $m$-quasi-Einstein manifold with $m>1$ and constant scalar curvature $R=\frac{2(m+2)\lambda}{m+1}.$ Then the following inequality holds
    \begin{equation*}
        L_{m+2}(|\nabla u|^2\left(Tr(P^3)-2m^3\rho^3)\right)\geq 2 (9m+7)\rho |\nabla u|^2\left(Tr(P^3)-2m^3\rho^3\right),
    \end{equation*} where $\rho=\frac{\lambda}{m+1}.$
\end{lemma}

\begin{proof}
We consider the level set $\Sigma=\Sigma(t)=u^{-1}(t),$ $0\leq t< u_{max},$ and an orthonormal frame $\{e_1,e_2,e_3,e_4\}$ for $M^4$ that diagonalizes the tensor $P$ so that $e_1=\frac{\nabla u}{|\nabla u|}$ and $\{e_2,e_3,e_4\}$ is a frame over $\Sigma(t).$ Moreover, we assume $\alpha,\beta,\gamma,\eta\in\{2,3,4\}$ and $i,j,k\in\{1,2,3,4\}$. Thereby, it follows from the Gauss equation that
\begin{eqnarray*}
    R_{\alpha\beta\gamma\eta}^{\Sigma}=R_{\alpha\beta\gamma\eta}+h_{\alpha\gamma}h_{\beta\eta}-h_{\alpha\eta}h_{\beta\gamma},
\end{eqnarray*}
which implies that 
\begin{eqnarray}\label{ric sigma}
    R_{\alpha\gamma}^{\Sigma}=R_{\alpha\gamma}-R_{\alpha 1\gamma 1}+H h_{\alpha\gamma}-h_{\alpha\beta}h_{\beta\gamma},
\end{eqnarray}
where $h$ and $H$ stand for the second fundamental form and the mean curvature, respectively. Besides, taking into account that $\rho=\frac{\lambda}{m+1}$ as well as
\begin{eqnarray*}
     Ric(\nabla u)=\rho\nabla u,\; P=Ric-\rho g,\; R=2(m+2)\rho,\; Tr(P)=2m\rho\;\;\;\;\mathrm{and}\;\;\;\;|P|^2=2m^2\rho^2,    
\end{eqnarray*} one deduces that
\begin{equation}
\label{eq RSigma}
    R^{\Sigma}=R-2\rho+H^2-|A|^2=2(m+1)\rho +H^2-|A|^2,
\end{equation} where $|A|^2$ is the norm of the second fundamental form.

Next, we are going to compute $h_{\alpha\beta}$ and $H.$ Indeed, by using (\ref{mqE}) in terms of $P,$ i.e., $\nabla^2 u=\frac{u}{m}(P-m\rho g)$, the second fundamental form yields
\begin{eqnarray}\label{eq scnd}
    h_{\alpha\beta}=\frac{\nabla_\alpha\nabla_\beta u}{|\nabla u|}=\frac{(P_{\alpha\beta}-m\rho g_{\alpha\beta})}{m\sqrt{b(u)}}u,
\end{eqnarray} where $b(u)=|\nabla u|^2.$ Furthermore, our assumption on the scalar curvature implies that $P_{11}=0$ and hence,
\begin{eqnarray}\label{eq H}
    H=\frac{Tr(P)-3m\rho}{m\sqrt{b(u)}}u=-\frac{\rho u}{\sqrt{b(u)}}.
\end{eqnarray} In particular, we have from \eqref{eq scnd} that
\begin{eqnarray}
\label{eq A}
    |A|^2&=&\frac{|P|^2-2m\rho Tr(P)+3m^2\rho^2}{m^2 b(u)}u^2 =\frac{\rho^2 u^2}{b(u)}.
\end{eqnarray} Substituting \eqref{eq H} and \eqref{eq A} into \eqref{eq RSigma} yields $R^{\Sigma}=2(m+1)\rho.$

Proceeding, we are going to deal with the Riemannian curvature tensor of $\Sigma.$ In fact, since $\Sigma$ has dimension $3,$ its curvature tensor can be expressed as
\begin{eqnarray*}
    R_{\alpha\beta\gamma\eta}^{\Sigma}=(R_{\alpha\gamma}^{\Sigma}g_{\beta\eta}+R_{\beta\eta}^{\Sigma}g_{\alpha\gamma}-R_{\alpha\eta}^{\Sigma}g_{\beta\gamma}-R_{\beta\gamma}^{\Sigma}g_{\alpha\eta})-\frac{R^{\Sigma}}{2}(g_{\alpha\gamma}g_{\beta\eta}-g_{\alpha\eta}g_{\beta\gamma}).
\end{eqnarray*} This jointly with \eqref{ric sigma} gives
\begin{eqnarray*}
    R_{\alpha\beta\alpha\beta}^{\Sigma}&=& R_{\alpha\alpha}^{\Sigma}+R_{\beta\beta}^{\Sigma}-\frac{R^{\Sigma}}{2}\\
    &=&R_{\alpha\alpha}-R_{\alpha 1\alpha 1}+H h_{\alpha\alpha}-h_{\alpha\alpha}^{2}+R_{\beta\beta}-R_{\beta 1\beta 1}+H h_{\beta\beta}-h_{\beta\beta}^{2}-(m+1)\rho\\
    &=&\mu_\alpha+\mu_\beta+2\rho-R_{\alpha 1\alpha 1}-R_{\beta 1 \beta 1}+H(h_{\alpha\alpha}+h_{\beta\beta})-h_{\alpha\alpha}^{2}-h_{\beta\beta}^{2}-(m+1)\rho,
\end{eqnarray*} where $\mu_\alpha=P(e_\alpha)$ and $h_{\alpha \beta}=0$ for $\alpha\neq \beta.$ Consequently, for fixed $\alpha\neq\beta$ again, by using the Gauss equation, Eqs. \eqref{eq scnd} and \eqref{eq H}, we then obtain
\begin{eqnarray*}
    R_{\alpha\beta\alpha\beta}&=&R_{\alpha\beta\alpha\beta}^{\Sigma}-h_{\alpha\alpha}h_{\beta\beta}+h^{2}_{\alpha\beta}\\
    &=&\mu_\alpha+\mu_\beta+2\rho-R_{\alpha 1\alpha 1}-R_{\beta 1 \beta 1}+H(h_{\alpha\alpha}+h_{\beta\beta})-h_{\alpha\alpha}^{2}\nonumber\\&&-h_{\beta\beta}^{2}-(m+1)\rho-h_{\alpha\alpha}h_{\beta\beta}\nonumber\\&=&\mu_\alpha+\mu_\beta+2\rho-R_{\alpha 1\alpha 1}-R_{\beta 1 \beta 1}-\frac{\rho(\mu_\alpha-m\rho+\mu_\beta-m\rho)u^2}{m b(u)}\\
    & &-\frac{(\mu_\alpha-m\rho)^2 u^2}{m^2 b(u)}-\frac{(\mu_\beta-m\rho)^2 u^2}{m^2 b(u)}-(m+1)\rho-\frac{(\mu_\beta-m\rho)(\mu_\alpha-m\rho)u^2}{m^2 b(u)}\\
    &=&\mu_\alpha+\mu_\beta+2\rho-R_{\alpha 1\alpha 1}-R_{\beta 1 \beta 1}-(m+1)\rho-\frac{m\rho(\mu_\alpha+\mu_\beta-2m\rho)u^2}{m^2 b(u)}\\
    & &-\frac{[\mu_{\alpha}^{2}-2m\rho(\mu_\alpha+\mu_\beta)+\mu_{\beta}^{2}+2m^2\rho^2]u^2}{m^2 b(u)}-\frac{[\mu_\beta\mu_\alpha-m\rho(\mu_\alpha+\mu_\beta)+m^2\rho^2]u^2}{m^2 b(u)},
\end{eqnarray*} which can be simplified as
\begin{eqnarray*}
    R_{\alpha\beta\alpha\beta}&=&\mu_{\alpha}+\mu_\beta-\frac{\rho(\mu_\alpha+\mu_\beta)u^2}{m b(u)}+\frac{2\rho(\mu_\alpha+\mu_\beta)u^2}{m b(u)}+\frac{\rho(\mu_\alpha+\mu_\beta)u^2}{m b(u)}\\
    & &+\frac{2\rho^2 u^2}{b(u)}-\frac{2\rho^2 u^2}{b(u)}-\frac{\rho^2 u^2}{b(u)}+2\rho-\frac{(\mu_{\alpha}^{2}+\mu_{\beta}^{2})u^2}{m^2 b(u)}\\
    & &-\frac{\mu_{\alpha}\mu_{\beta}u^2}{m^2 b(u)}-R_{\alpha 1\alpha 1}-R_{\beta 1\beta 1}-(m+1)\rho\\
    &=&\frac{(\mu_\alpha+\mu_\beta)(m b(u)+2\rho u^2)}{m b(u)}+\frac{\rho (2b(u)-\rho u^2)}{b(u)}\\
    & &-\frac{(\mu_{\alpha}^{2}+\mu_{\beta}^{2})u^2}{m^2 b(u)}-\frac{\mu_\alpha\mu_\beta u^2}{m^2 b(u)}-R_{\alpha 1\alpha 1}-R_{\beta 1\beta 1}-(m+1)\rho.
\end{eqnarray*}
Next, multiplying the previous expression by $\mu_\alpha\mu_\beta$ and summing over $\alpha$ and $\beta$, $\alpha\neq\beta$, we deduce that
\begin{eqnarray}\nonumber
    \sum_{\alpha\neq\beta}^{4}R_{\alpha\beta\alpha\beta}\mu_\alpha\mu_\beta&=&\frac{m b(u)+2\rho u^2}{m b(u)}\sum_{\alpha\neq\beta}^{4}(\mu_\alpha+\mu_\beta)\mu_\alpha\mu_\beta+\frac{\rho(2b(u)-\rho u^2)}{b(u)}\sum_{\alpha\neq\beta}^{4}\mu_{\alpha}\mu_{\beta}\\\nonumber
    & &-\frac{2u^2}{m^2 b(u)}\sum_{\alpha\neq\beta}^{4}\mu_{\alpha}^{3}\mu_{\beta}-\frac{u^2}{m^2 b(u)}\sum_{\alpha\neq\beta}^{4}\mu_{\alpha}^{2}\mu_{\beta}^{2}\\\label{Rab}
    & &-2\sum_{\alpha\neq\beta}^{4}R_{\alpha 1\alpha 1}\mu_{\alpha}\mu_{\beta}-(m+1)\rho\sum_{\alpha\neq\beta}^{4}\mu_{\alpha}\mu_{\beta}.
\end{eqnarray} At the same time, we derive an expression for each sum in \eqref{Rab}. To this end, we first observe that

\begin{eqnarray}\label{sums}
 \sum_{\alpha=2}^{4}\mu_\alpha=Tr(P)=2m\rho\;\;\;\;\mathrm{and}\;\;\;\;\sum_{\alpha=2}^{4}\mu_{\alpha}^{2}=|P|^2=2m^2\rho^2,   
\end{eqnarray} which implies that
\begin{eqnarray*}
    \sum_{\alpha\neq\beta}^{4}\mu_\alpha\mu_\beta&=&\sum_{\alpha=2}^{4}\sum_{\beta\neq\alpha}\mu_\alpha\mu_\beta=\sum_{\alpha=2}^{4}\mu_\alpha(Tr(P)-\mu_\alpha)=(Tr(P))^2-|P|^2=2m^2\rho^2,
\end{eqnarray*}
\begin{eqnarray*}
    \sum_{\alpha\neq\beta}^{4}(\mu_\alpha+\mu_\beta)\mu_\alpha\mu_\beta&=&2\sum_{\alpha=2}^{4}\sum_{\beta\neq\alpha}\mu_{\alpha}^{2}\mu_\beta=2\sum_{\alpha=2}^{4}\mu_{\alpha}^{2}(Tr(P)-\mu_\alpha)\\
    &=&2(Tr(P))|P|^2-2\sum_{\alpha=2}^{4}\mu_{\alpha}^{3}=8m^3\rho^3-2\sum_{\alpha=2}^{4}\mu_{\alpha}^{3},
\end{eqnarray*}
\begin{eqnarray*}
    \sum_{\alpha\neq\beta}^{4}\mu_{\alpha}^{3}\mu_\beta&=&\sum_{\alpha=2}^{4}\sum_{\beta\neq\alpha}\mu_{\alpha}^{3}\mu_\beta=\sum_{\alpha=2}^{4}\mu_{\alpha}^{3}(Tr(P)-\mu_\alpha)=\sum_{\alpha=2}^{4}2m\rho\mu_{\alpha}^{3}-\sum_{\alpha=2}^{4}\mu_{\alpha}^{4}
\end{eqnarray*} and
\begin{eqnarray*}
    \sum_{\alpha\neq\beta}^{4}\mu_{\alpha}^{2}\mu_{\beta}^{2}&=&\sum_{\alpha=2}^{4}\sum_{\beta\neq\alpha}\mu_{\alpha}^{2}\mu_{\beta}^{2}=\sum_{\alpha=2}^{4}\mu_{\alpha}^{2}(|P|^2-\mu_{\alpha}^{2})=4m^4\rho^4-\sum_{\alpha=2}^{4}\mu_{\alpha}^{4}.
\end{eqnarray*}
We also need to obtain an expression for $R_{\alpha 1\alpha 1}.$ From Eq. (4) of Lemma \ref{lemmafund}, one deduces that
\begin{eqnarray*}
    u(\nabla_i P_{jk}-\nabla_j P_{ik})\nabla_j u&=&m R_{ijkl}\nabla_l u\nabla_j u+m\rho(\nabla_i u g_{jk}-\nabla_j u g_{ik})\nabla_j u\\
    & &-(\nabla_i u P_{jk}-\nabla_j u P_{ik})\nabla_j u,
\end{eqnarray*} where we have used that $\lambda=(m+1)\rho.$ This combined with the fact that $P_{jk}\nabla_j u=0$ and
\begin{eqnarray*}
    \nabla_i P_{jk}\nabla_j u&=&\nabla_i(P_{jk}\nabla_j u)-P_{jk}\nabla_i\nabla_j u=-\frac{u}{m}P_{jk}(P_{ij}-m\rho g_{ij})
\end{eqnarray*} allow us to infer
\begin{eqnarray*}
    R_{ijkl}\nabla_l u\nabla_j u&=&-\rho(\nabla_i u\nabla_k u-|\nabla u|^2g_{ik})-\frac{|\nabla u|^2}{m}P_{ik}\\
    & &-\frac{u^2}{m^2}P_{jk}(P_{ij}-m\rho g_{ij})-\frac{u}{m}\nabla_j P_{ik}\nabla_j u.
\end{eqnarray*} By taking $i=k=\alpha$ and multiplying the last expression by $\frac{|\nabla u|^2}{|\nabla u|^2},$ we obtain
\begin{eqnarray*}
    R_{\alpha 1\alpha 1}|\nabla u|^2&=&\rho|\nabla u|^2-\frac{|\nabla u|^2}{m}\mu_{\alpha}-\frac{u^2}{m^2}P_{j\alpha}(P_{\alpha j}-m\rho g_{\alpha j})-\frac{u}{m}\nabla_1 P_{\alpha\alpha}|\nabla u|\\
    &=&\frac{(m\rho-\mu_\alpha)|\nabla u|^2}{m}-\frac{u^2}{m^2}\mu_{\alpha}^{2}+\frac{\rho u^2}{m}\mu_\alpha-\frac{u}{m}\nabla_1 P_{\alpha\alpha}|\nabla u|.
\end{eqnarray*}
Consequently,
\begin{eqnarray*}
    \sum_{\alpha\neq\beta}^{4}R_{\alpha 1\alpha 1}\mu_{\alpha}\mu_{\beta}&=&\sum_{\alpha=2}^{4}\sum_{\beta\neq\alpha}R_{\alpha 1\alpha 1}\mu_\alpha\mu_{\beta}=\sum_{\alpha=2}^{4}R_{\alpha 1\alpha 1}\mu_\alpha(Tr(P)-\mu_\alpha)\\
    &=&\frac{1}{|\nabla u|^2}\sum_{\alpha=2}^{4}\left[\frac{(m\rho-\mu_\alpha)b(u)+\rho\mu_\alpha u^2}{m}-\frac{u^2}{m^2}\mu_{\alpha}^{2}-\frac{u}{m}\nabla_1 P_{\alpha\alpha}|\nabla u|\right]\mu_\alpha(2m\rho-\mu_\alpha)\\
    &=&\frac{1}{|\nabla u|^2}\sum_{\alpha=2}^{4}\frac{(2m^2\rho^2\mu_\alpha-3m\rho\mu_{\alpha}^{2}+\mu_{\alpha}^{3})b(u)}{m}+\frac{1}{|\nabla u|^2}\sum_{\alpha=2}^{4}\frac{(2m\rho^2\mu_{\alpha}^{2}-\rho\mu_{\alpha}^{3})u^2}{m}\\
    & &-\frac{u^2}{m^2|\nabla u|^2}\sum_{\alpha=2}^{4}(2m\rho\mu_{\alpha}^{3}-\mu_{\alpha}^{4})-\frac{u}{m|\nabla u|^2}\sum_{\alpha=2}^{4}\nabla_1 P_{\alpha\alpha}|\nabla u|\left(2m\rho\mu_{\alpha}-\mu_{\alpha}^{2}\right).
\end{eqnarray*}
In order to conclude this step, observe that
\begin{eqnarray*}
    \nabla_1 Tr(P^3)=3\sum_{\alpha=2}^{4}(\nabla_1 P_{\alpha\alpha})\mu_{\alpha}^{2}\;\;\;\;\mathrm{and}\;\;\;\;0=\nabla_1 |P|^2=2\sum_{\alpha=2}^{4}(\nabla_1 P_{\alpha\alpha})\mu_\alpha,
\end{eqnarray*} which combined with \eqref{sums} gives
\begin{eqnarray*}
    \sum_{\alpha\neq\beta}^{4}R_{\alpha1 \alpha 1}\mu_\alpha\mu_\beta&=&\frac{4m^3\rho^3-6m^3\rho^3}{m}+\frac{1}{m}\sum_{\alpha=2}^{4}\mu_{\alpha}^{3}+\frac{4m^3\rho^4 u^2}{m b(u)}-\frac{\rho u^2}{m b(u)}\sum_{\alpha=2}^{4}\mu_{\alpha}^{3}\\
    & &-\frac{u^2}{m^2 b(u)}\sum_{\alpha=2}^{4}(2m\rho\mu_{\alpha}^{3}-\mu_{\alpha}^{4})+\frac{\nabla u(Tr(P^3))u}{3m b(u)}\\
    &=&-2m^2\rho^3+\frac{4m^2\rho^4 u^2}{b(u)}+\frac{\nabla u(Tr(P^3))u}{3m b(u)}+\frac{b(u)-3\rho u^2}{m b(u)}\sum_{\alpha=2}^{4}\mu_{\alpha}^{3}+\frac{u^2}{m^2 b(u)}\sum_{\alpha=2}^{4}\mu_{\alpha}^{4}.
\end{eqnarray*}

Returning to Eq. \eqref{Rab}, we then have 
\begin{eqnarray*}
    \sum_{\alpha\neq\beta}^{4}R_{\alpha\beta\alpha\beta}\mu_\alpha\mu_\beta&=&\frac{m b(u)+2\rho u^2}{m b(u)}\left(8m^3\rho^3-2\sum_{\alpha=2}^{4}\mu_{\alpha}^{3}\right)+\frac{\rho(2b(u)-\rho u^2)}{b(u)}\cdot 2m^2\rho^2\\
    & &-\frac{2u^2}{m^2 b(u)}\left(\sum_{\alpha=2}^{4}2m\rho\mu_{\alpha}^{3}-\sum_{\alpha=2}^{4}\mu_{\alpha}^{4}\right)-\frac{u^2}{m^2 b(u)}\left(4m^4\rho^4-\sum_{\alpha=2}^{4}\mu_{\alpha}^{4}\right)\\
    & &-2m^2(m+1)\rho^3+4m^2\rho^3-\frac{8m^2\rho^4 u^2}{b(u)}-\frac{2\nabla u(Tr(P^3))u}{3mb(u)}\\
    & &-\frac{2b(u)-6\rho u^2}{mb(u)}\sum_{\alpha=2}^{4}\mu_{\alpha}^{3}-\frac{2u^2}{m^2 b(u)}\sum_{\alpha=2}^{4}\mu_{\alpha}^{4}.
\end{eqnarray*} Simplifying terms, we infer
\begin{eqnarray*}
    \sum_{\alpha\neq\beta}^{4}R_{\alpha\beta\alpha\beta}\mu_\alpha\mu_\beta&=&\frac{8m^3\rho^3b(u)+16m^2\rho^4 u^2+4m^2\rho^3 b(u)-2m^2\rho^4 u^2-2m^2(m-1)\rho^3 b(u)}{b(u)}\\
    & &-\frac{4m^2\rho^4 u^2}{b(u)}-\frac{8m^2\rho^4 u^2}{b(u)}-\frac{2\nabla u(Tr(P^3))u}{3mb(u)}\\
    & &-\frac{2m b(u)+4\rho u^2+4\rho u^2+2b(u)-6\rho u^2}{mb(u)}\sum_{\alpha=2}^{4}\mu_{\alpha}^{3}+\frac{2u^2+u^2-2u^2}{m^2 b(u)}\sum_{\alpha=2}^{4}\mu_{\alpha}^{4}\\
    &=&\frac{6m^2(m+1)\rho^3 b(u)+2m^2\rho^4 u^2}{b(u)}-\frac{2\nabla u(Tr(P^3))u}{3mb(u)}\\
    & &-\frac{2(m+1)b(u)+2\rho u^2}{mb(u)}Tr(P^3)+\frac{u^2}{m^2 b(u)}\sum_{\alpha=2}^{4}\mu_{\alpha}^{4}\\
    &=&\frac{2m^2\rho^3[3(m+1)b(u)+\rho u^2]}{b(u)}-\frac{2\nabla u(Tr(P^3))u}{3mb(u)}\\
    & &-\frac{2[(m+1)b(u)+\rho u^2]}{mb(u)}Tr(P^3)+\frac{u^2}{m^2 b(u)}\sum_{\alpha=2}^{4}\mu_{\alpha}^{4}.
\end{eqnarray*}

On the other hand, it follows from (\ref{eqi2}) that
\begin{eqnarray*}
    2u|\nabla P|^2+4(m+1)\rho u|P|^2+4u P_{ik}R_{jikl}P_{jl}=0
\end{eqnarray*} and hence,
\begin{eqnarray*}
    u|\nabla P|^2=-2(m+1)\rho u|P|^2+2uP_{ik}R_{ijkl}P_{jl}.
\end{eqnarray*} Plugging this fact into \eqref{idP2} yields
\begin{eqnarray}
\label{lmn120p}
    u L_{m+2}(Tr(P^3))&\geq& 8(m+1)\rho u Tr(P^3)-3m\rho u|\nabla P|^2-16m^3(m+1)\rho^4 u\nonumber\\
    &=&8(m+1)\rho u Tr(P^3)+6m(m+1)\rho^2 u|P|^2-6m\rho uP_{ik}R_{ijkl}P_{jl}\nonumber\\
    & &-16m^3(m+1)\rho^4 u\nonumber\\
    &=&8(m+1)\rho u Tr(P^3)-4m^3(m+1)\rho^4 u-\frac{12m^3\rho^4 u[3(m+1)b(u)+\rho u^2]}{b(u)}\nonumber\\
    & &+\frac{4\rho\nabla u(Tr(P^3))u^2}{b(u)}+\frac{12\rho u[(m+1)b(u)+\rho u^2]}{b(u)}Tr(P^3)-\frac{6\rho u^3}{m b(u)}\sum_{\alpha=2}^{4}\mu_{\alpha}^{4}\nonumber\\
    &=&\frac{4\rho u[5(m+1)b(u)+3\rho u^2]}{b(u)}Tr(P^3)-\frac{6\rho u^3}{m b(u)}\sum_{\alpha=2}^{4}\mu_{\alpha}^{4}\nonumber\\
    & &+\frac{4\rho\nabla u(Tr(P^3))u^2}{b(u)}-\frac{4m^3\rho^4 u[10(m+1)b(u)+3\rho u^2]}{b(u)}.
\end{eqnarray}

From \eqref{sums}, it is known that $\mu_2,$ $\mu_3,$ $\mu_4$ and $Tr(P)$ satisfy the hypothesis of Corollary A.1 in \cite{ChengZhou} and therefore,
\begin{eqnarray*}
    \sum_{\alpha=2}^{4}\mu_{\alpha}^{4}=-\frac{10m^4\rho^4}{3}+\frac{8m\rho}{3}Tr(P^3).
\end{eqnarray*} Substituting the above equality into (\ref{lmn120p}), we infer
\begin{eqnarray}\nonumber
    uL_{m+2}(Tr(P^3))&\geq&\frac{4\rho u[5(m+1)b(u)+3\rho u^2]}{b(u)}Tr(P^3)+\frac{20m^3\rho^5 u^3}{b(u)}-\frac{16\rho^2 u^3}{b(u)}Tr(P^3)\\\nonumber
    & &+\frac{4\rho\nabla u(Tr(P^3))u^2}{b(u)}-\frac{4m^3\rho^4 u[10(m+1)b(u)+3\rho u^2]}{b(u)}\\\nonumber
    &=&\frac{4\rho u[5(m+1)b(u)-\rho u^2]}{b(u)}Tr(P^3)+\frac{4\rho \nabla u(Tr(P^3))u^2}{b(u)}\\\nonumber
    & &-\frac{4m^3\rho^4 u[10(m+1)b(u)-2\rho u^2]}{b(u)}\\\label{principalK}
    &=&\frac{4\rho u[5(m+1)b(u)-\rho u^2]}{b(u)}(Tr(P^3)-2m^3\rho^3)+\frac{4\rho \nabla u(Tr(P^3))u^2}{b(u)}.
\end{eqnarray}

Finally, we recall that, by using (\ref{transnormal}) and (\ref{eqb}), the potential function of a quasi-Einstein manifold with constant scalar curvature is transnormal satisfying 
\begin{eqnarray*}
    b(u)=|\nabla u|^2=\frac{\mu}{m-1}-\frac{R+(m-n)\lambda}{m(m-1)}u^2=\rho(u_{max}^{2}-u^2).
\end{eqnarray*} Hence,
\begin{eqnarray}
\label{lkjn670}
    uL_{m+2}(b(u)(Tr(P^3)-2m^3\rho^3))&=&ub(u)L_{m+2}(Tr(P^3))+2u\langle\nabla b(u),\nabla(Tr(P^3))\rangle\nonumber\\
    & &+(Tr(P^3)-2m^3\rho^3)uL_{m+2}(b(u))\nonumber\\
    &=&ub(u)L_{m+2}(Tr(P^3))-4\rho u^2\nabla u(Tr(P^3))\nonumber\\
    &&+(Tr(P^3)-2m^3\rho^3)(-2\rho u^2\Delta u-2\rho u|\nabla u|^2\nonumber\\&& -(m+2)2u\rho|\nabla u|^2)\nonumber\\
    &=&ub(u)L_{m+2}(Tr(P^3))-4\rho u^2\nabla u(Tr(P^3))\nonumber\\
    & &-2\rho u\left(-2\rho u^2+(m+3)b(u)\right)(Tr(P^3)-2m^3\rho^3),
\end{eqnarray} where we have used that $\Delta u=-2\rho u$ and 
\begin{eqnarray*}
    L_{a}(f)=u^{-a}div(u^a\nabla f)=\Delta f+au^{-1}\langle\nabla u,\nabla f\rangle,\,\hbox{for}\,\,\, a\neq 0\,\,\,\hbox{and}\,\,\, f\in C^{\infty}(M).
\end{eqnarray*} Comparing \eqref{principalK} with (\ref{lkjn670}) gives
\begin{eqnarray*}
    uL_{m+2}\left(|\nabla u|^2(Tr(P^3)-2m^3\rho^3)\right)\geq2 (9m+7)\rho u |\nabla u|^2\left(Tr(P^3)-2m^3\rho^3\right),
\end{eqnarray*} as asserted. 
\end{proof}

\vspace{0.20cm}
We are now prepared to present the proof of Theorem \ref{theodim4}, which we restate here for convenience.

\begin{theorem}[Theorem \ref{theodim4}]
\label{theodim4proof}
 Let $(M^4,\,g,\,u,\,\lambda)$ be a nontrivial simply connected compact $4$-dimensional $m$-quasi-Einstein manifold with boundary and $m>1.$ Then $M^4$ has constant scalar curvature $R=2\frac{(m+2)}{(m+1)}\lambda$ if and only if it is isometric, up to scaling, to the product space $\mathbb{S}^{2}_{+}\times\mathbb{S}^2$ with the product metric. 
\end{theorem}

\begin{proof}
We already know that $Tr(P)=2m\rho$ and $|P|^2=2m^2\rho^2$, that is, 
\begin{equation}
\label{hjgfl}
|P|^2=\frac{1}{2}(Tr(P))^2.
\end{equation} Hence, since $\mu_1=0,$ Lemma \ref{lemsum} ensures that all eigenvalues $\mu_\alpha,$ $\alpha=1,2,3,4$, of $P$ are nonnegative. 

Define the function $$h:=|\nabla u|^2(Tr(P^3)-2m^3\rho^3).$$ In particular, from (\ref{eqTYU8}) and the fact that $\mu_\alpha,$ $\alpha=1,2,3,4,$ are all nonnegative, we see that $h$ is nonnegative on $M.$ Since $M$ is compact with boundary $\partial M,$ by performing integration by parts, we deduce
    
    \begin{eqnarray}
    \label{divtheo}
        \int_M L_{m+2}(h) dV_{m+2}&=&\int_{M} u^{-(m+2)}div (u^{m+2}\nabla h) dV_{m+2}=\int_M div(u^{m+2}\nabla h) dV\nonumber\\&=&-\int_{\partial M}u^{m+2}\left\langle\nabla h,\frac{\nabla u}{|\nabla u|}\right\rangle dS= 0,
    \end{eqnarray} where we have used the facts that $u$ vanishes on $\partial M,$ $dV_{m+2}=u^{m+2} dV$ is the weighted measure and the second-order operator $L_{a}$ $(a\in\mathbb{R})$ is given by 
    \begin{eqnarray*}
        L_a(f)=u^{-a}div(u^a\nabla f)=\Delta f+au^{-1}\langle \nabla u,\nabla f\rangle,
    \end{eqnarray*} for any $f\in C^{\infty}(M).$ 
    
On the other hand, it follows from Lemma \ref{lemKAp} that
    \begin{eqnarray}\label{eqh}
        2(9m+7)\rho h-L_{m+2}(h)\leq 0.
    \end{eqnarray} So, upon integrating (\ref{eqh}) over $M,$ we use (\ref{divtheo}) in order to infer
    
    \begin{eqnarray*}
2(9m+7)\rho\int_{M}  h\,dV_{m+2}\leq 0.
    \end{eqnarray*}
 Of which, one obtains that $$h=|\nabla u|^2(Tr(P^3)-2m^3\rho^3)=0.$$ Since $u$ is nonconstant and $g$ is analytical, we conclude that $Tr(P^3)-2m^3\rho^3\equiv 0$ on $M.$ Together with Eq. (\ref{eqTYU8}), this implies $\mu_2\mu_3\mu_4=0,$ and thus at least one among $\mu_2,$ $\mu_3$ and $\mu_4$ vanishes. Assume $\mu_2=0.$ Then, by using (\ref{hjgfl}), we deduce $\mu_1=\mu_2=0$ and $\mu_3=\mu_4=m\rho$. 

Returning to the Ricci tensor, we find that it has exactly two distinct eigenvalues, each of multiplicity two:
\begin{eqnarray*}
    \lambda_1=\lambda_2=\frac{\lambda}{m+1}\;\mathrm{and}\;\lambda_3=\lambda_4=\lambda,
\end{eqnarray*} where $Ric(e_i)=\lambda_i,$ for $i=1,\,2,\,3,\,4.$ In particular, the Ricci tensor $Ric$ is parallel. By the first contracted second Bianchi identity ($\nabla_l R_{ijkl}=\nabla_j R_{ik}-\nabla_i R_{jk}$), it follows that the curvature tensor is harmonic. We can therefore apply \cite[Corollary 1.14]{Petersen-Chenxu} to conclude that $M^4$ is rigid. By Proposition \ref{prop-hpw}, $M^4$ is covered by the product $\mathbb{S}^{2}_{+}\times \mathbb{S}^{2}.$ Since $M^4$ is simply connected, Theorem 54.6 in \cite{Munkres} ensures that the covering map is a bijective local isometry, hence a global isometry. Thus, $M^4$ is isometric, up to scaling, to the product space $\mathbb{S}^{2}_{+}\times\mathbb{S}^2.$ This completes the proof of the theorem. 
\end{proof}

\subsection{Proof of Corollary \ref{corA}}

\begin{proof} 
The result follows from Theorem \ref{theo1}, Remark \ref{remL}, Proposition \ref{propK11}, Theorem \ref{theo3}, and Theorem \ref{theodim4}.
\end{proof}

\section{Appendix}
For the reader's convenience, we collect here some useful facts about the distance function that were employed in the proofs of the main results. Let $M$ be a complete Riemannian manifold and $N\subset int(M)$ a properly immersed submanifold of $M.$ Let $\pi: \nu N\to N$ be the normal bundle. There is an induced connection $\nabla^\nu $ on $\nu N$ and a decomposition of tangent bundle $T(\nu N)$ as
$$ T(\nu N)=\mathcal{H} \oplus \mathcal{V}, $$
where $\mathcal{V}_\xi:=\ker(d\pi)_\xi$ and $\mathcal{H}_\xi$ consists of all tangent vectors to parallel sections passing through $\xi$. If $\alpha: (-\delta, \delta)\to \nu N$ is a smooth curve representing $v\in T(\nu N)$, then
$v^\mathcal{H}=(\pi\circ\alpha)'(0)$ and $v^\mathcal{V}=(\frac{\nabla^\nu}{\partial s}\alpha)(0)=v-v^\mathcal{H}.$ Thus, $\mathcal{H}_\xi$ and $\mathcal{V}_\xi$ are isomorphic to $T_{\pi(\xi)}N $ and $\nu_{\pi(\xi)} N,$ respectively. This decomposition induces a natural Riemannian metric on $T(\nu N)$ such that $\pi$ becomes a Riemannian submersion; for more details, see \cite[p. 11]{Ball}. With this notation in hand, we state the following lemma.

\begin{lemma}[\cite{Ball}]
Let $\alpha: (-\delta, \delta)\to \nu N$ be a smooth curve representing $v\in T(\nu N)$. Define
\[J(t):=\left.\frac{\partial}{\partial s}\right|_{s=0}\exp_{\pi\circ\alpha(s)}(t\alpha(s)) .
\]
Then $J(t)$ is a Jacobi field along the geodesic $\gamma(t)=\exp(t\alpha(0)) $ and
$$J(0)=v^\mathcal{H},\,\,\,\,J(1)=(d\exp)_{\alpha(0)}(v)\,\,\,\,\,\,\,\hbox{and}\,\,\,\,\,\,\,J'(0)=v^\mathcal{V}+A_{\alpha(0)}v^\mathcal{H}.$$
Here, $A_\eta$ stands for the shape operator with respect to normal vector $\eta.$
\end{lemma}

Proceeding, let $UN$ be the unit normal bundle of $N$ equipped with volume element $d\theta dp,$ where $dp$ denotes the volume element of $N$ and $d\theta$ is the volume element of unit sphere $\mathbb{S}^{n-k-1}_p$ in $\nu_pN.$ Thereby, we may define
$\Phi: (0,a)\times UN\to M\backslash N$ by $\Phi(r, \theta)=\exp (r\theta).$ In particular, if $M$ has boundary $\partial M,$ we take $a\leq \frac{1}{2}dist(N,\,\partial M).$

Along the normal geodesic $\gamma_\theta(r)= \exp(r\theta),$ we can choose a parallel orthonormal base $\{e_1(r),\ldots, e_n(r)\}$ such that
$$A_\theta e_i(0)=\lambda_i, \textrm{ for } i=1,\cdots, k-1,\,\,\,\,\,\textrm{ and }\,\,\,\,\,e_n=\partial r=\gamma_\theta'(r).$$
Hence, $J_i(r)=(d\Phi)_{(r, \theta)}(e_i)$, $i=1,2,\cdots, n,$ must satisfy
\[\begin{split}
&J_i''(t)+R(\gamma_\theta'(t), J_i(t))\gamma_\theta'(t)=0, \textrm{ for } i=1, \ldots, k;\\
&J_i(0)=e_i(0), \textrm{ for } i=1, \cdots, k;\\
&J'_i(0)=\lambda_ie_i(0), \textrm{ for } i=1,\cdots, k;\\
&J_i(0)=0, \textrm{ for } i=k+1, \cdots, n;\\
&J_i'(0)=e_i(0), \textrm{ for } i=k+1, \cdots, n.
\end{split}
\]
Next, we consider the following notation
\[\begin{split}
&J_{ij}=\langle J_i, e_j\rangle, \textrm{ for } i=1, \cdots, k;\\
&K_{ij}=\langle R(\gamma_\theta, e_i)\gamma_\theta, e_j\rangle,\textrm{ for } i=1, \cdots, k;\\
&\mathcal{A}=\textrm{diag} (\lambda_1, \cdots, \lambda_n).
\end{split}
\]
Also consider $\mathcal{J}:=(J_{ij})_{(k-1)\times(k-1)}$ and $\mathcal{K}:=(K_{ij})_{(k-1)\times(k-1)}.$ With these notations, one obtains that
\[
\left\{
\begin{split}
&\mathcal{J}''+\mathcal{K}\mathcal{J}=0;\\
&\mathcal{J}(0)=\textrm{diag}\left(\mathcal{I}_{k\times k}, \mathcal{O}_{(n-k-1)\times(n-k-1)}\right);\\
&\mathcal{J}'(0)=\textrm{diag}\left(\mathcal{A}, \mathcal{I}_{(n-k-1)\times(n-k-1)}\right),
\end{split}
\right.
\] where $\mathcal{O}$ and $\mathcal{I}$ denote the zero matrix and the identity matrix, respectively. If $\gamma_\theta|_{[0,r]}$ does not contain focal points, then $\mathcal{J}$ is invertible on $(0,r)$. Next, let $\sigma(x)$ be the distance function from $N.$ Therefore, $\sigma(\gamma_\theta(r))=r,$ provided that $r\in (0, r_\theta).$ Moreover, by denoting $\mathcal{U}_{ij}(r):=\nabla^2 \sigma(e_i, e_j)(\gamma_\theta(r))$ and taking into account that $\nabla^2  \sigma(J_i, J_j)=\langle J_i',J_j\rangle,$ we get the following lemma.

\begin{lemma}[\cite{Ball}]
\label{lemmaD} Let $N$ be a proper submanifold in $M$. Then for any $\theta\in\nu N$, along the normal geodesic $\gamma_\theta(r)= \exp(r\theta)$, the Hessian of the distance function $\sigma(x)=dist (x, N)$ satisfies
\[
\left\{
\begin{aligned}
&\mathcal{U}'+\mathcal{U}^2+\mathcal{K}=0,\\
&\mathcal{U}=\begin{pmatrix}
\mathcal{A}_\theta& \\
& \frac1{r}\mathcal{I}
\end{pmatrix}+r\begin{pmatrix}
-\mathcal{A}_\theta^2-\mathcal{K}_{11}(0)& \mathcal{V}_{12}\\
\mathcal{V}_{21}& \mathcal{V}_{22}
\end{pmatrix}+O(r^2),
\end{aligned}
\right.
\]
where $\mathcal{U}=\nabla^2 \sigma |_{\{\gamma'_\theta(r)\}^\perp}$, $\mathcal{K}=\mathcal{K}_\theta=R(\gamma'_\theta, \ldots)\gamma'_\theta$ and $\mathcal{A}_\theta$ is the shape operator of $N$ with respect to $\theta$. In particular, the mean curvature $H(\theta,r)$ of the level sets of $\sigma$ at $\gamma_\theta(r)$ satisfies
\begin{equation}
\label{eqHtheta}
H(\theta,r)=\textrm{tr}(\mathcal{A}_\theta)+\frac{n-k-1}r+O(r)
\end{equation}
and
\begin{equation}
\nabla^2 \frac{\sigma^2}2(\gamma_\theta(r))=\begin{pmatrix}
r\mathcal{A}_\theta& \\
& \mathcal{I}_{(n-k)\times(n-k)}
\end{pmatrix}+O(r^2).
\end{equation}
Moreover, at $N$, the function $\sigma^2$ has two eigenvalues $0$ and $2$ of multiplicities $m$ and $n-k,$ respectively.
\end{lemma}

In the sequel, we are going to present the proof of the following algebraic ine\-qua\-lity. 

\begin{lemma}
\label{lemsum}
Let $a_{1}\geq \ldots \geq a_{n}$ be $n\geq 2$ real numbers. Then

$$a_{i}a_{j}\geq \frac{b}{2(n-1)},$$ where $b=\left(\sum_{i=1}^n a_{i}\right)^2 - (n-1)\sum_{i=1}^{n}a_{i}^2.$ In particular, if $b\geq 0,$ then either all $a_{i}\geq 0$ or all $a_{i}\leq 0.$
\end{lemma}

\begin{proof}
The case $n=2$ is straightforward. Now, for $n>2,$ notice that
$$\left(\sum_{i=1}^{n}a_{i}\right)^2 =\left(\sum_{i=1}^{n-1}a_{i}\right)^2+2a_{n}\sum_{i=1}^{n-1}a_{i}+a_{n}^2.$$ Hence, we see that 

$$(n-1)\sum_{i=1}^n a_{i}^2 +b=\left(\sum_{i=1}^{n-1}a_{i}\right)^2+2a_{n}\sum_{i=1}^{n-1}a_{i}+a_{n}^2,$$ so that  
$$(n-1)\sum_{i=1}^{n-1} a_{i}^2 + (n-2)a_{n}^2 +b=\left(\sum_{i=1}^{n-1}a_{i}\right)^2+2a_{n}\sum_{i=1}^{n-1}a_{i}.$$ In view of this, one obtains that
$$(n-2)\sum_{i=1}^{n-1}a_{i}^2 + (n-2)a_{n}^2 -2a_{n}\sum_{i=1}^{n-1}a_{i}+b=\left(\sum_{i=1}^{n-1}a_{i}\right)^2-\sum_{i=1}^{n-1}a_{i}^2,$$ which implies that

$$2\sum_{i<j\leq n-1}a_{i}a_{j}= (n-2)\sum_{i=1}^{n-1}a_{i}^2 +(n-2)a_{n}^2 -2a_{n}\sum_{i=1}^{n-1}a_{i}+b.$$ Rearranging terms, one sees that 

$$(n-2)a_{n}^2 -2\left(\sum_{i=1}^{n-1}a_{i}\right)a_{n}+\left[(n-2)\sum_{i=1}^{n-1}a_{i}^2 +b -2\sum_{i<j\leq n-1}a_{i}a_{j}\right]= 0.$$ Of which, we have 

\begin{eqnarray*}
\left(\sum_{i=1}^{n-1}a_{i}\right)^2
&= & (n-2)\left[(n-2)\sum_{i=1}^{n-1}a_{i}^2 +b -2\sum_{i<j\leq n-1}a_{i}a_{j}\right]+ (n-2)^2\left(a_{n}-\frac{1}{n-2}\sum_{i=1}^{n-1}a_{i}\right)^2\nonumber\\
&=& (n-2)\left[(n-1)\sum_{i=1}^{n-1}a_{i}^2 -\left(\sum_{i=1}^{n-1}a_{i}\right)^2 +b\right]+ (n-2)^2\left(a_{n}-\frac{1}{n-2}\sum_{i=1}^{n-1}a_{i}\right)^2.
\end{eqnarray*} Consequently,
\begin{equation}
\label{kjnzq01}
\left(\sum_{i=1}^{n-1}a_{i}\right)^2 \geq (n-2)\sum_{i=1}^{n-1}a_{i}^2 +\frac{n-2}{n-1}b.
\end{equation} Moreover, if equality holds in (\ref{kjnzq01}), then $a_{n}=\frac{1}{n-2}\sum_{i=1}^{n-1}a_{i}.$ Now, it suffices to repeat an analogous process $n-2$ times in order to obtain the asserted inequality. 
\end{proof}

	\noindent{\bf Conflict of Interest:} There is no conflict of interest to disclose.
	
	\

\noindent{\bf Data Availability:} Not applicable.


\begin{thebibliography}{BB}


\bibitem{Ambrozio} Ambrozio, L.: \emph{On static three-manifolds with positive scalar curvature}. J. Diff. Geom. 107 (2017) 1--45.

\bibitem{BGKW} Bahuaud, E., Gunasekaran, S., Kunduri, K. and Woolgar, E.: {\it Static near-horizon geometries and rigidity of quasi-Einstein manifolds}. Lett. Math. Phys. 112 (2022) 116. 

\bibitem{BGKW2} Bahuaud, E., Gunasekaran, S., Kunduri, K. and Woolgar, E.: {\it Rigidity of quasi-Einstein metrics: The incompressible case}. Lett. Math. Phys. 114 (2024) 8.

\bibitem{bakry} Bakry, D. and \'Emery, M.: {\it Diffusions Hypercontractives. S\'eminaire de probabilit\'es, XIX, 1983/84}. Lecture Notes in Math., vol. 1123, Springer, Berlin, (1985) 177--206.

\bibitem{Ball} Ballmann, W.: {\it Riccati equation and volume estimates}, Lectures Notes - https://people.mpim-bonn.mpg.de/hwbllmnn/notes.html, (2016)


\bibitem{Ernani2} Barros, A., Batista, R. and Ribeiro Jr., E.: {\it Bounds on volume growth of geodesic balls for Einstein warped products}. Proc. Amer. Math. Soc. 143 (2015) 4415--4422.


\bibitem{BRS14} Barros, A., Ribeiro Jr, E. and Silva, J.: {\it Uniqueness of quasi-Einstein metrics on 3-
dimensional homogeneous manifolds}. Diff. Geom. and its App. 35 (2014) 60-73.

\bibitem{BRR} Batista, R., Ranieri, M. and Ribeiro Jr, E.: {\it Remarks on complete noncompact Einstein warped pro\-ducts.} Comm. Anal. Geom. 28 (3) (2020) 547--563.

	\bibitem{Bergery} B\'erard Bergery, L.: \textit{Sur de nouvelles vari\'etes riemanniennes d'Einstein}. Public. Institut E. Cartan. 4 (Nancy) 1--60 (1982).
	
	\bibitem{Besse} Besse, A. L.: \emph{Einstein Manifolds}. Springer-Verlag, Berlin (1987).


\bibitem{Bohm} B\"ohm, C.: {\it Inhomogeneous Einstein metrics on low-dimensional spheres and other low-dimensional spaces}. Invent. Math. 134 (1) (1998) 145--176.

\bibitem{Bohm2} B\"ohm, C.: {\it Non-compact cohomogeneity one Einstein manifolds}. Bull. Soc. Math. France. 127 (1999) 135--177.

\bibitem{BM} Borghini, S. and Mazzieri, L.: {\it On the mass of static metrics with positive cosmological constant - I}. Class. Quantum Grav. 35 (2018) 125001.

\bibitem{BMC} Borghini, S., Chru\'sciel, P. and Mazzieri, L.: {\it On the uniqueness of Schwarzchild-de Sitter spacetime}. Arch. Rational Mech. Anal. 247, 22 (2023).


\bibitem{Cao1} Cao, H.-D.: {\it Recent progress on Ricci solitons}, Recent advances in geometric analysis,
Adv. Lect. Math. (ALM), vol. 11, Int. Press, Somerville, MA, (2010), pp. 1--38.

\bibitem{CCZ} Cao, H.-D., Chen, B.-L. and Zhu, X.-P.: {\it Recent developments on Hamilton’s Ricci flow}, Surveys in differential geometry. Vol. XII. Geometric flows, Surv. Differ. Geom., vol. 12, Int. Press, Somerville, MA, (2008) 47--112.

\bibitem{CC} Cao, H.-D. and Chen, Q.: \emph{On Bach-flat gradient shrinking Ricci solitons}. Duke Math. J. 162 (2013) 1149-1169.

\bibitem{CaseShuWei} Case, J., Shu, Y.-S. and Wei, G.: {\it Rigidity of quasi-Einstein metrics}. Differ. Geom. Appl. 29 (2011) 93--100.

\bibitem{CaseT} Case, J.: {\it Smooth metric measure spaces, quasi-Einstein metrics, and tractors.} Cent. Eur. J. Math. 10 (2012) 1733-1762.

\bibitem{CaseP} Case, J.: {\it Smooth metric measure spaces and quasi-Einstein metrics}. Intern. J. Math. 23, no. 10 (2012) 1250110.

\bibitem{catino} Catino, G.: \emph{A note on four-dimensional (anti-)self-dual quasi-Einstein manifolds}. Diff. Geom. Appl. 30 (6) (2012) 660--664.

\bibitem{CMMM} Catino, G., Mantegazza, C., Mazzieri, L. and Rimoldi, M.: {\it Locally conformally flat quasi-Einstein manifolds}. J. Reine Angew. Math. (Crelle's Journal). 675 (2013) 181--189. 

\bibitem{CH} Chen, Q. and He, C.: \emph{On Bach flat warped product Einstein manifold}. Pacific J. Math. 265 (2), (2013) 313--326.

\bibitem{ChengZhou} Cheng, X. and Zhou, D.: \emph{Rigidity of four-dimensional gradient shrinking Ricci solitons}. J. Reine Angew. Math. (Crelle's Journal). 802 (2023) 255--274.

\bibitem{CRZJGEA} Cheng, X., Ribeiro Jr, E. and Zhou, D.: {\it Volume growth estimates for Ricci solitons and quasi-Einstein manifolds}. J. Geom. Anal. 32 (2022) 62.

\bibitem{CDPR} Costa, J., Di\'ogenes, R., Pinheiro, N. and Ribeiro Jr, E.: {\it Geometry of static perfect fluid space-time}. Class. Quantum Grav. 40 (2023) 205012.

\bibitem{compact} Di\'ogenes, R. and Gadelha, T.: \emph{Compact quasi-Einstein manifolds with boundary}. Math. Nachr. 295 (2022) 1690--1708.

\bibitem{remarks} Di\'ogenes, R., Gadelha, T. and Ribeiro Jr, E.: \emph{Remarks on quasi-Einstein manifolds with boundary}.  Proc. Amer. Math. Soc. 150 (2022) 351--363.

\bibitem{Fl-Gr} Fern\'andez-L\'opez, M. and Garc\'ia-R\'io, E.: \emph{On gradient Ricci solitons with constant scalar curvature}. Proc. Amer. Math. Soc. 144 (2016) 369--378. 

\bibitem{Fl-Gr2} Fern\'andez-L\'opez, M. and Garc\'ia-R\'io, E.: \emph{A remark on compact Ricci solitons}. Math. Z. 340 (2008) 893--896. 

\bibitem{GT} Ge, J. and Tang, Z.: {\it Geometry of isoparametric hypersurfaces in Riemannian manifolds}. Asian J. Math. 18 (2014) 117--126.

\bibitem{GT2} Ge, J. and Tang, Z.: {\it Isoparametric functions and exotic spheres}. J. Reine Angew. Math. (Crelle's Journal). 683 (2013) 161--180.

\bibitem{GS1999} Gompf, R. and Stipsicz, A.: {\it $4$-manifolds and Kirby Calculus}. Graduate Studies in Mathematics, vol. 20. American Mathematical Society, Providence (1999).

\bibitem{Hamilton} Hamilton, R.: {\it The formation of singularities in the Ricci flow}, Surveys in differential geometry, Vol. II (Cambridge, MA, 1993), Int. Press, Cambridge, MA, (1995) 7--136

\bibitem{He-Petersen-Wylie2012} He, C., Petersen, P. and Wylie, W.: \emph{On the classification of warped product Einstein metrics}. Comm. Anal. Geom. 20 (2012) 271--311.

\bibitem{Petersen-Chenxu} He, C., Petersen, P. and Wylie, W.: \emph{Warped product Einstein metrics over spaces with constant scalar curvature}. Asian J. Math. 18 (2014) 159--190.

\bibitem{Ivey} Ivey, T.: {\it New examples of complete Ricci solitons}, Proc. Amer. Math. Soc. 122 (1994), no. 1, 241--245.

\bibitem{kk} Kim, D.-S. and Kim, Y. H.: \emph{Compact Einstein warped product spaces with nonpositive scalar curvature}. Proc. Amer. Math. Soc. 131 (2003) 2573--2576.

\bibitem{Kobayashi} Kobayashi, O.: {\it A differential equation arising from scalar curvature function}. J. Math. Soc. Japan. 34 (1982) 665--75.

\bibitem{lafontaine} Lafontaine, J.: {\it Sur la geomerie d'une generalisation de l’equation differentielle d’Obata}. J. Math. Pures Appl. 62 (1983) 63--72.

\bibitem{Lu} L\"u, H, Page, D. and Pope, C.: {\it New inhomogeneous Eintein metrics on sphere bundles over Einstein-K\"ahler manifolds}. Phys. Lett. B 593(1-4) (2004) 218--226.

\bibitem{MR} Mastrolia, P. and Rimoldi, M.: {\it Some triviality results for quasi-Einstein manifolds and Einstein warped products}. Geom. Dedic. 169 (2014) 225--237.

\bibitem{MP} Mazzeo, R. and Pacard, F.: {\it Foliations by constant mean curvature tubes}. Comm. Anal. Geom. 13 (2005) 633--670.

\bibitem{Miyaoka} Miyaoka, R.: {\it Transnormal functions on a Riemannian manifold}. Differ. Geom. Appl. 31 (2013) 130--139.

\bibitem{Munkres} Munkres, J.: {\it Topology}. 2nd. ed. Prentice Hall, Inc. (2000).

\bibitem{Naber} Naber, A.: {\it Noncompact shrinking four solitons with nonnegative curvature}. J. Reine Angew. Math. (Crelle's Journal). 645 (2010) 125--153.

\bibitem{NW} Ni, L. and Wallach, N.: {\it On a classification of gradient shrinking solitons}. Math. Res. Lett. 15 no. 5 (2008) 941--955.

\bibitem{O'Neil} O'Neill, B.: \emph{Semi-Riemannian Geometry with applications to Relativity}. Academic Press, Nova York, (1983).

\bibitem{Perelman} Perelman, G.: {\it Ricci flow with surgery on three manifolds}, ArXiv:math.DG/0303109 (2003).

\bibitem{PW} Petersen, P. and William W.: {\it Rigidity of gradient Ricci solitons}. Pacific J. Math. 241 (2009) 329--345.

\bibitem{Qian} Qian, Z.: {\it Estimates for weighted volumes and applications}. Quart. J. Math. 48 (1997) 235--242.

\bibitem{QY1} Qing, J. and Yuan, W.: {\it On scalar curvature rigidity of vacuum static spaces}. Math. Ann. 365 no. 3-4 (2016) 1257--1277.

\bibitem{QY} Qing, J. and Yuan, W.: \emph{A note on vacuum static spaces and related problems}. J. Geom. Phys. 74 (2013) 18--27.

\bibitem{Rei} Reilly, R.: {\it Geometric applications of the solvability of Neumann problems on a Riemannian manifold}. Arch. Rational Mech. Anal. (1) 75 (1980) 23--29.

\bibitem{Ernani_Keti} Ribeiro Jr, E. and Tenenblat, K.: {\it Noncompact quasi-Einstein manifolds conformal to a Euclidean space}. Math. Nachr. 294 (2021) 132--144.

\bibitem{Rimoldi} Rimoldi, M.: {\it A remark on Einstein warped products}. Pacific J. Math. 252 (2011) 207--218.

\bibitem{Rimoldi2} Rimoldi, M.: {\it Rigidity results for Lichnerowicz Bakry-Emery Ricci tensors}. Thesis (PhD in Mathematics), Universita degli Studi di Milano (2011).

\bibitem{LFWang} Wang, L.: {\it On noncompact $\tau$-quasi-Einstein metrics}. Pacific J. Math. 254 (2011) 449--464.

\bibitem{LFWang2} Wang, L.: {\it Potential function estimates for quasi-Einstein metrics}. J. Functional Analysis. 267 (2014) 1986--2004.

\bibitem{Wang} Wang, Q. M.: \emph{Isoparametric functions on Riemannian manifolds. I.} Math. Ann. 277, no. 4 (1987) 639--646.

\bibitem{WW} Wei, G. and Wylie, W.: {\it Comparison geometry for the Bakry-Emery Ricci tensor}. J. Diff. Geom. 83 (2009) 337--405.

\bibitem{WW2} Wei, G. and Wylie, W.: {\it Comparison geometry for the smooth metric measure spaces}. In: Proceedings of the 4th International Congress of Chinese Mathematicians, v. 2 (2007) 191--202.

\bibitem{Wylie} Wylie, W.: {\it Rigidity of compact static near-horizon geometries with negative cosmological constant}. Lett. Math. Phys. (2023) 113:29.

\bibitem{WylieF} Wylie, W.: {\it Complete shrinking Ricci solitons have finite fundamental group}. Proc. Amer. Math. Soc. 136 (2008) 1803--1806.

	\end{thebibliography}
\end{document}